\documentclass[10pt,a4paper]{amsart}
\usepackage{amsfonts}
\usepackage{amssymb,amsmath,amsthm,stmaryrd}
\usepackage{multirow,bigdelim}
\usepackage[american]{babel}

\title{A Nullstellensatz for sequences over $\mathbb{F}_p$}
\author{\'{E}ric Balandraud}
\author{Benjamin Girard}
\thanks{IMJ, \'{E}quipe Combinatoire et Optimisation, Universit\'{e} Pierre et Marie Curie~(Paris $6$), $4$ place Jussieu, 75005 Paris, France, email: \texttt{eric.balandraud@imj-prg.fr|benjamin.girard@imj-prg.fr}. 
Research supported by the French ANR Project "CAESAR" No. ANR-12-BS01-0011.}

\newtheorem{theo}{Theorem}[section]
\newtheorem{lem}[theo]{Lemma}
\newtheorem{prop}[theo]{Proposition}
\newtheorem{cor}[theo]{Corollary}

\def\le{\leqslant}
\def\ge{\geqslant}

\begin{document}

\begin{abstract} 
Let $p$ be a prime and let $A=(a_1,\dots,a_\ell)$ be a sequence of nonzero elements in $\mathbb{F}_p$. 
In this paper, we study the set of all $0$-$1$ solutions to the equation
$$a_1x_1 + \dots + a_\ell x_\ell = 0.$$
We prove that whenever $\ell \ge p$, this set actually characterizes $A$ up to a nonzero multiplicative constant, which is no longer true for $\ell < p$. 
The critical case $\ell=p$ is of particular interest. 
In this context, we prove that whenever $\ell=p$ and $A$ is nonconstant, the above equation has at least $p-1$ minimal $0$-$1$ solutions, thus refining a theorem of Olson.
The subcritical case $\ell=p-1$ is studied in detail also.
Our approach is algebraic in nature and relies on the Combinatorial Nullstellensatz as well as on a Vosper type theorem. 
\end{abstract}

\maketitle

\section{Introduction}
The study of the existence and variety of $0$-$1$ solutions to linear equations over a finite Abelian group is a central topic in zero-sum combinatorics.
Promoted by applications in algebraic number theory, many results have been proved in this area since the seminal paper of Erd\H{o}s, Ginzburg and Ziv \cite{EGZ61}. 
The interested reader is referred to \cite{GeroRuzsa09,GeroKoch05} for an exposition and many references on this subject.

\medskip
Let $p$ be a prime and let $A=(a_1,\dots,a_\ell)$ be a sequence of nonzero elements in $\mathbb{F}_p$. 
In this paper, we study the set $\mathcal{S}_A$ of all $0$-$1$ solutions to the equation
\begin{equation}
\label{maineq}
a_1x_1 + \dots + a_\ell x_\ell = 0.
\end{equation}
Equivalently, one has
$$\mathcal{S}_A = A^{\perp} \cap \{0,1\}^\ell$$ 
where $A^{\perp}=\{x \in \mathbb{F}^{\ell}_p : A \cdot x =0\}$ and $A \cdot x$ denotes the usual inner product of $A$ and $x$ over $\mathbb{F}^{\ell}_p$.
This very definition readily makes of $\mathcal{S}_A$ an ambivalent object that can be expressed as the intersection of an algebraic structure, namely a hyperplane of $\mathbb{F}^{\ell}_p$, and the $\ell$-cube $\{0,1\}^\ell$, which has a more combinatorial flavor. 

\medskip
Seen as a set of vectors, $\mathcal{S}_A$ will be studied through the subspace $\langle \mathcal{S}_A \rangle$ of $A^{\perp}$. 
For this reason, we will set  
$$\dim(A)=\dim(\langle S_A \rangle).$$
In particular, since $\dim(A^{\perp})=\ell-1$, it can easily be noticed that $\dim(A) \le \ell - 1$ always holds.
Alternatively, every element of $\mathcal{S}_A$ can be seen as a subset of $[1,\ell]$. 
Therefore, it is not surprising that $\mathcal{S}_A$ also has set theoretic properties. 
For instance, $\mathcal{S}_A$ is easily seen to be closed under disjoint union. 
It is also closed under complement if and only if the sum $\sigma(A)$ of all elements in $A$ is equal to zero. 
In what follows, we will be particularly interested in the number of \em minimal \em elements in $\mathcal{S}_A \setminus \{0\}$ ordered by inclusion.

\medskip
In this paper, we address three problems which appear to be closely related. 
First, we determine the value of $\dim(A)$ for any sequence $A$ of $\ell \ge p-1$ nonzero elements in $\mathbb{F}_p$. 
Then, we apply our results to solve the reconstruction problem of knowing whether $\mathcal{S}_A$ characterizes $A$ up to a nonzero multiplicative constant. 
Finally, we study the number of minimal elements in $\mathcal{S}_A$.

\medskip
The main idea behind our results is the following. 
For every $\ell \ge p-1$, the equality $\dim(A)=\ell-1$ holds true except for a very limited number of exceptional sequences which can be fully determined.
Our first theorem deals with large values of $\ell$, namely $\ell > p$.

\begin{theo}
\label{main theorem long >p} 
Let $p$ be a prime and let $A=(a_1,\dots,a_\ell)$ be a sequence of $\ell > p$ nonzero elements in $\mathbb{F}_p$. 
Then
$$\dim(A)=\ell-1.$$
\end{theo}

Consequently, whenever $A$ is a sequence of $\ell > p$ nonzero elements in $\mathbb{F}_p$, the hyperplane $A^{\perp}$ has a basis consisting solely of elements in $\mathcal{S}_A$, and is thus fully characterized by its intersection with the $\ell$-cube. This solves the reconstruction problem for large values of $\ell$.

\begin{cor} 
\label{long >p characterisation}
Let $p$ be a prime. 
Let also $A,B$ be two sequences of $\ell > p$ elements in $\mathbb{F}_p$ such that the elements of $A$ are nonzero.
Then $\mathcal{S}_A \subset \mathcal{S}_B$ if and only if $A$ and $B$ are collinear. 
\end{cor}
Theorem \ref{main theorem long >p} also implies a lower bound on the number of minimal elements in $\mathcal{S}_A$.  
Indeed, one can notice that any basis of $A^{\perp}$ consisting of elements in $\mathcal{S}_A$ can be turned into a basis of $A^{\perp}$ consisting exclusively of minimal elements in $\mathcal{S}_A$ (see Proposition \ref{mimimini}). This yields the following result.

\begin{cor} 
\label{long >p minimales}
Let $p$ be a prime and let $A=(a_1,\dots,a_\ell)$ be a sequence of $\ell > p$ nonzero elements in $\mathbb{F}_p$. 
Then $\mathcal{S}_A$ contains at least $\ell-1$ minimal elements.
\end{cor}

In the case $\ell=p$, there are exceptional sequences $A$ such that $\dim(A) < \ell-1$.
However, such exceptions are very rare, since they have to be highly structured. 
The following analogue of Theorem \ref{main theorem long >p} is our first result in this direction.

\begin{theo}
\label{main theorem long p} 
Let $p$ be a prime and let $A=(a_1,\dots,a_p)$ be a sequence of $p$ nonzero elements in $\mathbb{F}_p$. 
Then one of the following statements holds.
\begin{itemize}
\item[$(i)$] $\dim(A)=1$ and there exists $r \in \mathbb{F}^*_p$ such that
$$(a_1,\dots,a_p)=(r,\dots,r).$$
\item[$(ii)$] $\dim(A)=p-2$ and there exist $t \in [1,p-3]$, $\sigma\in\mathfrak{S}_p$, $r \in \mathbb{F}^*_p$ such that
$$(a_{\sigma(1)},\dots,a_{\sigma(p)})=(\underbrace{r,\dots,r}_{t},\underbrace{-r,\dots,-r}_{p-2-t},-(t+1)r,-(t+1)r).$$
\item[$(iii)$] $\dim(A)=p-1$.
\end{itemize}
\end{theo}

Therefore, in all cases where $\dim(A)=p-1$, the same argument as above readily implies that $A^{\perp}$ has a basis consisting of elements in $\mathcal{S}_A$. 
In addition, the two exceptional cases $(i)$ and $(ii)$ correspond to highly structured sequences, for which we can easily check by hand that $\mathcal{S}_A$ characterizes $A^{\perp}$ indeed.
This answers the reconstruction problem affirmatively in the case $\ell=p$.

\begin{cor} 
\label{reconstruction l=p}
Let $p$ be a prime. Let also $A,B$ be two sequences of $p$ nonzero elements in $\mathbb{F}_p$.
Then $\mathcal{S}_A=\mathcal{S}_B$ if and only if $A$ and $B$ are collinear. 
\end{cor}

In his last paper \cite{Ol}, Olson proved a conjecture of Erd\H{o}s stating that if  $A$ is a nonconstant sequence of $p$ nonzero elements in $\mathbb{F}_p$, then $\mathcal{S}_A$ contains at least $p-1$ nonzero elements. 
In this respect, Theorem \ref{main theorem long p} actually yields a strong version of the original conjecture of Erd\H{o}s.   
More precisely, this conjecture still holds when we restrict ourselves to counting minimal elements in $\mathcal{S}_A$ only. 
This gives the following refinement of Olson's theorem. 

\begin{cor} 
\label{long p minimales}
Let $p$ be a prime and let $A=(a_1,\dots,a_p)$ be a nonconstant sequence of $p$ nonzero elements in $\mathbb{F}_p$. 
Then $\mathcal{S}_A$ contains at least $p-1$ minimal elements. 
\end{cor}

In the case $\ell=p-1$, the situation becomes a little more involved but similar results can still be proved.
On the one hand, all sequences $A$ with $\sigma(A) \neq 0$ can easily be handled (see Lemma \ref{dim of non zero-sum sequences}).
On the other hand, the following theorem fully describes the case of sequences $A$ with $\sigma(A)=0$, which will be called \em zero-sum \em sequences. 

\begin{theo}
\label{main theorem long p-1}
Let $p$ be a prime and let $A=(a_1,\dots,a_{p-1})$ be a zero-sum sequence of $p-1$ nonzero elements in $\mathbb{F}_p$. 
Then one of the following statements holds.
\begin{itemize}
\item[$(i)$] $\dim(A)=1$ and there exist $\sigma\in\mathfrak{S}_{p-1}$, $r \in \mathbb{F}^*_p$ such that
$$(a_{\sigma(1)},\dots,a_{\sigma(p-1)})=(r,\dots,r,2r).$$
\item[$(ii)$] $\dim(A)=p-4$ and there exist $\sigma\in\mathfrak{S}_{p-1}$, $r \in \mathbb{F}^*_p$ such that
$$(a_{\sigma(1)},\dots,a_{\sigma(p-1)})=(\underbrace{r,\dots,r}_{p-5},-r,2r,2r,2r).$$
\item[$(iii)$] $\dim(A)=p-3$ and there exist $t\in[0,p-6]$, $\sigma\in\mathfrak{S}_{p-1}$, $r \in \mathbb{F}^*_p$ such that
$$(a_{\sigma(1)},\dots,a_{\sigma(p-1)})=(\underbrace{r,\dots,r}_{t},\underbrace{-r,\dots,-r}_{p-4-t},2r,-(t+3)r,-(t+3)r).$$
\item[$(iv)$] $\dim(A)=p-3$ and there exist $t\in[1,p-4]$, $\sigma\in\mathfrak{S}_{p-1}$, $r \in \mathbb{F}^*_p$ such that
$$(a_{\sigma(1)},\dots,a_{\sigma(p-1)})=(\underbrace{r,\dots,r}_{t},\underbrace{-r,\dots,-r}_{p-3-t},-(t+1)r,-(t+2)r).$$
\item[$(v)$] $p=7$, $\dim(A)=p-3=4$ and there exists $\sigma\in\mathfrak{S}_6$ such that
\[(a_{\sigma(1)},\dots,a_{\sigma(6)})=(-1,1,-2,2,-3,3).\]
\item[$(vi)$] $\dim(A)=p-2$.
\end{itemize}
\end{theo}

It turns out that $\ell=p$ actually is the critical case for our reconstruction problem, since there are sequences $A$ of $p-1$ nonzero elements in $\mathbb{F}_p$ for which $\mathcal{S}_A$ does not fully characterize $A^{\perp}$. However, we can prove there are essentially two families of such exceptions. 

\begin{cor}
\label{reconstruction long p-1}
Let $p$ be a prime. Let also $A,B$ be two sequences of $p-1$ nonzero elements in $\mathbb{F}_p$.
Then $\mathcal{S}_A=\mathcal{S}_B$ if and only if $A$ and $B$ are either collinear or have one of the following forms.
\begin{itemize}
\item[$(i)$] There exist $\sigma\in\mathfrak{S}_{p-1}$, $r \in \mathbb{F}^*_p$, $\lambda \in \mathbb{F}^*_p$ such that
$$(a_{\sigma(1)},\dots,a_{\sigma(p-1)})=(r,\dots,r,r,2r)$$
and
$$(b_{\sigma(1)},\dots,b_{\sigma(p-1)})=(\lambda r,\dots,\lambda r,2\lambda r,\lambda r).$$
\item[$(ii)$] There exist $t \in [1,p-4]$, $\sigma\in\mathfrak{S}_{p-1}$, $r \in \mathbb{F}^*_p$, $\lambda \in \mathbb{F}^*_p$ such that
$$(a_{\sigma(1)},\dots,a_{\sigma(p-1)})=(\underbrace{r,\dots,r}_{t},\underbrace{-r,\dots,-r}_{p-3-t},-(t+1)r,-(t+2)r)$$
and
$$(b_{\sigma(1)},\dots,b_{\sigma(p-1)})=(\underbrace{\lambda r,\dots,\lambda r}_{t},\underbrace{-\lambda r,\dots,-\lambda r}_{p-3-t},-(t+2)\lambda r,-(t+1)\lambda r).$$
\item[$(iii)$] $p=7$ and there exist $\sigma\in\mathfrak{S}_6$, $\lambda \in \mathbb{F}^*_7$ such that
$$(a_{\sigma(1)},\dots,a_{\sigma(6)})=(-1,1,-2,2,-3,3)$$
and
$$(b_{\sigma(1)},\dots,b_{\sigma(6)})=(-\lambda,\lambda,3\lambda,-3\lambda,2\lambda,-2\lambda).$$
\end{itemize}
\end{cor}

As for the number of minimal elements in $\mathcal{S}_A$ when $\ell=p-1$, the following result easily follows from Theorem \ref{main theorem long p-1}.

\begin{cor} 
\label{long p-1 minimales}
Let $p$ be a prime and let $A=(a_1,\dots,a_{p-1})$ be a sequence of $p-1$ nonzero elements in $\mathbb{F}_p$.
 Then $\mathcal{S}_A$ contains at least $p-3$ minimal elements unless $A$ has one of the following forms.  
\begin{itemize}
\item[$(i)$] There exists $r \in \mathbb{F}^*_p$ such that
$$(a_1,\dots,a_{p-1})=(r,\dots,r).$$
\item[$(ii)$] There exist $\sigma\in\mathfrak{S}_{p-1}$ and $r \in \mathbb{F}^*_p$ such that
$$(a_{\sigma(1)},\dots,a_{\sigma(p-1)})=(r,\dots,r,2r).$$
\end{itemize}
\end{cor}

The outline of the paper is as follows. 
In Section \ref{section : tools}, we recall some useful tools. 
In Section \ref{section : long >p}, we introduce the notion of a ratio set, which makes it possible to deduce our Theorem \ref{main theorem long >p} from the Combinatorial Nullstellensatz.
In Section \ref{section : long p}, we prove Theorem \ref{main theorem long p} and its corollaries. 
Our method is similar in nature to the one used in Section \ref{section : long >p}, except that a Vosper type theorem comes into play.
In Section \ref{section : long p-1}, this approach is further refined to solve our three problems in the case $\ell = p-1$.
Finally, in Section \ref{section : conclusion}, we extend our results to the affine setting.

\section{Some tools}
\label{section : tools}
In this section, we present a series of results we will use throughout the paper. 
As already mentioned, our method relies on the Combinatorial Nullstellensatz \cite{Al}, a useful algebraic tool having a wide range of applications in Combinatorics.

\begin{theo}[Alon \cite{Al}]
\label{Combnull}
Let $\mathbb{F}$ be an arbitrary field, and let $P$ be a polynomial in $\mathbb{F}[X_1,\dots,X_d]$. 
Suppose that the degree of $P$ is $\sum^d_{i=1} t_i$ and the coefficient of $\prod^d_{i=1} X_i^{t_i}$ is nonzero. 
Then, for any subsets $S_1,\dots,S_d$ of $\mathbb{F}$ with $|S_i| > t_i$, there exists $(s_1,\dots,s_d)\in S_1\times\dots\times S_d$ such that
$$P(s_1,\dots,s_d)\neq 0.$$
\end{theo}

In Additive Combinatorics, this theorem provides a unified way to estimate the cardinality of all kinds of sumsets in $\mathbb{F}_p$. 
These objects often are restrictions of the \em Minkowski sum \em
$$A+B\, =\, \{a+b :\, a \in A, b\in B\}\,$$
of two nonempty subsets $A$ and $B$ of $\mathbb{F}_p$.
A collection of such estimates can be found in \cite[Chapter 9]{TaoVu}. 
For instance, a first corollary of Theorem \ref{Combnull} is the well-known Cauchy-Davenport Theorem.

\begin{theo}[Cauchy-Davenport \cite{Ca,Dav,Da2}]
Let $p$ be a prime and let $A,B$ be two nonempty subsets of $\mathbb{F}_p$. 
Then
$$|A+B| \ge \min\left\{p,|A|+|B|-1\right\}.$$
\end{theo}
Let us mention that equality in Cauchy-Davenport Theorem only holds for highly structured sets, which are characterized by Vosper's Theorem \cite{Vo,Vo2}.
To be more precise, the essential part of Vosper's Theorem states that whenever $|A|,|B| \ge 2$, the relation
$$|A+B|=|A|+|B|-1< p-1$$
holds if and only if $A$ and $B$ are arithmetic progressions with the same difference.

Another central object of interest in Additive Combinatorics is the set of subsums 
$$\Sigma(A) \, = \, \sum_{i=1}^\ell\{0,a_i\}$$
of a sequence $A=(a_1,\dots,a_{\ell})$ over $\mathbb{F}_p$.
Thanks to the properties of some binomial determinants, a sharp lower bound for the cardinality of the set of subsums has been recently deduced from the Combinatorial Nullstellensatz \cite[Theorem 8]{EB}.
To state this result, we require the following definition. 
For any element $a \in \mathbb{F}^*_p$, the total number of appearances of $a$ and $-a$ in the sequence $A$ will be called the \em multiplicity \em of the pair $\{a,-a\}$ in $A$.

\begin{theo}
\label{ThB} 
Let $p$ be a prime and let $A=(a_1,\dots,a_\ell)$ be a sequence of $\ell \ge 1$ nonzero elements in $\mathbb{F}_p$. 
Let also $\ell_1\ge \ell_2 \ge \cdots \ge \ell_k$ be all multiplicities of the pairs $\{a,-a\}$ in $A$ listed in decreasing order.
Then
$$\left|\Sigma(A)\right|\ge \min\left\{p,1+\sum_{i=1}^ki\ell_i\right\}.$$
\end{theo}

In this paper, we will also need structural information on sequences whose set of subsums is relatively small. 
For this purpose, the following two Vosper type results can easily be obtained from Theorem \ref{ThB}.

\begin{lem}
\label{Vosper} 
Let $p$ be a prime and let $A=(a_1,\dots,a_\ell)$ be a sequence of $\ell \ge 1$ nonzero elements in $\mathbb{F}_p$ such that 
$$\left|\Sigma(A)\right|=\ell+1 < p.$$ 
Then there exists $r\in \mathbb{F}^*_p$ such that $a_i \in \{\pm r\}$ for all $i \in [1,\ell]$.  
\end{lem}

\begin{proof}
Denoting by $\ell_1\ge \ell_2\ge \cdots \ge \ell_k$ the multiplicities of the pairs $\{a,-a\}$ in $A$, we readily obtain $\sum_{i=1}^k\ell_i=\ell$.
Since $\ell+1 < p$,  Theorem \ref{ThB} yields
$$1+\sum_{i=1}^ki\ell_i \le \ell + 1,$$
so that $\sum_{i=1}^k(i-1)\ell_i = 0$. Then $\ell_1=\ell$, which is the desired result. 
\end{proof}

Going one step further, we can also prove the following lemma.

\begin{lem}
\label{theinv} 
Let $p$ be a prime and let $A=(a_1,\dots,a_\ell)$ be a sequence of $\ell \ge 2$ nonzero elements in $\mathbb{F}_p$ such that 
$$\left|\Sigma(A)\right|=\ell+2 < p.$$
Then one of the following two statements holds.
\begin{itemize}
\item[$(i)$] There exist $r\in \mathbb{F}^*_p$ and $i_0\in[1,\ell]$ such that $a_{i_0}\in\{\pm 2r\}$ and $a_i \in \{\pm r\}$ for all $i \neq i_0$.
\item[$(ii)$] One has $\ell=2$.
\end{itemize}
\end{lem}

\begin{proof}
Denoting by $\ell_1\ge \ell_2 \ge \cdots \ge \ell_k$ the multiplicities of the pairs $\{a,-a\}$ in $A$, we readily obtain $\sum_{i=1}^k\ell_i=\ell$.
Since $\ell + 2 <p$, Theorem \ref{ThB} gives
$$1+\sum_{i=1}^ki\ell_i \le \ell + 2,$$
so that $\sum_{i=1}^k(i-1)\ell_i \le 1$. 
Then either $\ell_1=\ell$ or $(\ell_1,\ell_2)=(\ell-1,1)$.
If $\ell_1=\ell$, then there exists $r\in \mathbb{F}^*_p$ such that $a_i \in \{\pm r\}$ for all $i \in [1,\ell]$,
which implies $\left|\Sigma(A)\right|=\ell+1$, a contradiction.
Now, suppose that $(\ell_1,\ell_2)=(\ell-1,1)$ and $\ell>2$. Then there exist $r \in \mathbb{F}^*_p$ and $i_0 \in [1,\ell]$ such that $a_i \in \{\pm r\}$ for all $i \neq i_0$. 
It follows that $\Sigma(A \smallsetminus (a_{i_0}))$ and $a_{i_0}+\Sigma(A \smallsetminus (a_{i_0}))$ both are arithmetic progressions of length $\ell>2$ with same difference $r$.
Thus, for the equality $|\Sigma(A)|=|\{0,a_{i_0}\}+\Sigma(A \smallsetminus (a_{i_0}))|=\ell+2$ to hold, we must have $a_{i_0}\in\{\pm 2r\}$.
\end{proof}

\section{Ratio sets and long sequences}
\label{section : long >p}

Let $p$ be a prime and let $A$ be a sequence of $\ell \ge 1$ nonzero elements in $\mathbb{F}_p$.
In this section, we describe first our general approach to proving $\dim(A)=\ell-1$.
Then, we prove Theorem \ref{main theorem long >p} and its corollaries. 

\subsection{Ratio sets}
Let $p$ be a prime. Let also $A=(a_1,\dots,a_{\ell})$ and $B=(b_1,\dots,b_{\ell})$ be two sequences of $\ell \ge 1$ elements in $\mathbb{F}_p$ such that the elements of $A$ are nonzero.
For every $\lambda \in \mathbb{F}_p$, we consider the set 
$$I_{\lambda}=\left\{i \in [1,\ell] : b_i/a_i=\lambda\right\}.$$
Any element $\lambda \in \mathbb{F}_p$ such that $I_{\lambda} \neq \emptyset$ will be called a \em ratio \em associated with $(A,B)$. 
Now, let $\lambda_1,\dots,\lambda_d$ be the pairwise distinct ratios associated with $(A,B)$.
For every $i \in [1,d]$, we define the subsequence 
$$S_i=(a_j : j\in I_{\lambda_i}).$$
In particular, one has
$$\sum_{i=1}^d|S_i|=\ell.$$
Finally, for every $i \in [1,d]$, we introduce the \em ratio set \em
$$\Sigma_i=\Sigma(S_i).$$
It follows from Cauchy-Davenport Theorem that 
\begin{equation}
|\Sigma_i|\ge\min\{p,|S_i|+1\}.
\end{equation}

In this paper, our main argument to prove that $\dim(A)=\ell-1$ is as follows. 
First, note that such an equality holds if and only if for every $B \in \langle \mathcal{S}_A \rangle^\perp$, $A$ and $B$ are collinear, that is to say, for every sequence $B$ of $\ell$ elements in $\mathbb{F}_p$ such that $\mathcal{S}_A \subset \mathcal{S}_B$, one has $d=1$. 
In order to prove that $d=1$ indeed, properties of the ratio sets $\Sigma_i$ will be deduced from the Combinatorial Nullstellensatz. 
Before stating our first result in this direction, we observe that the ratio sets $\Sigma_i$ associated with $(A,B)$ readily have the following interesting feature. 

\begin{lem}
\label{x-x}
Let $p$ be a prime. 
Let $A,B$ be two sequences of $\ell \ge 1$ elements in $\mathbb{F}_p$ such that the elements of $A$ are nonzero and $\mathcal{S}_A \subset \mathcal{S}_B$. 
Let also $\lambda_1,\dots,\lambda_d$ be the ratios associated with $(A,B)$.
Then $\Sigma_i\cap (-\Sigma_j)=\{0\}$ for any two distinct elements $i,j \in [1,d]$.
\end{lem}

\begin{proof} 
By definition, we readily have $0 \in \Sigma_i$ for all $i \in [1,d]$.
Now, assume there is an element $s \in \Sigma_i \cap (-\Sigma_j)$ such that $s \neq 0$. 
By definition of $\Sigma_i$ and $\Sigma_j$, there exist $J_i \subset I_{\lambda_i}$ and $J_j \subset I_{\lambda_j}$ such that $\sum_{k \in J_i}a_k=s=-\sum_{k\in J_j}a_k$.
Since $\mathcal{S}_A \subset \mathcal{S}_B$, then
$$\sum_{k \in J_i} a_k + \sum_{k \in J_j} a_k =0$$ 
yields
$$\sum_{k \in J_i} b_k + \sum_{k \in J_j} b_k =0$$
which is equivalent to
$$\sum_{k \in J_i} \lambda_i a_k + \sum_{k \in J_j} \lambda_j a_k =0$$
so that $$(\lambda_i-\lambda_j)s =0,$$ 
which is a contradiction.
\end{proof}

On the one hand, an immediate corollary of Lemma \ref{x-x} is that whenever $d \ge 2$, one has $|\Sigma_i|<p$ and $|\Sigma_i|\ge |S_i|+1$ for all $i \in [1,d]$, so that
\begin{equation}\label{LB}
\sum_{i=1}^d(|\Sigma_i|-1)\ge \ell.
\end{equation}
On the other hand, the ratio sets also have the following general property, which will be useful throughout the paper.

\begin{theo}
\label{prop1} 
Let $p$ be a prime. 
Let $A,B$ be two sequences of $\ell \ge 1$ elements in $\mathbb{F}_p$ such that the elements of $A$ are nonzero and 
$\mathcal{S}_A \subset \mathcal{S}_B$.
Let also $\lambda_1,\dots,\lambda_d$ be the ratios associated with $(A,B)$. 
If $d \ge 2$, then
$$\sum_{i=1}^d(|\Sigma_i|-1)\le p.$$
In addition, $\sum_{i=1}^d(|\Sigma_i|-1)=p$ implies $\sum_{i=1}^d\lambda_i(|\Sigma_i|-1)=0$ in $\mathbb{F}_p$.
\end{theo}

\begin{proof}
Let $\lambda_1,\dots,\lambda_d$ be the ratios associated with $(A,B)$ and assume $d \ge 2$. 
We shall consider the polynomial
$$P(X_1,\dots,X_d)=\left(\sum_{i=1}^d\lambda_i X_i\right)\left(\left(\sum_{i=1}^dX_i\right)^{p-1}-1\right),$$
which has degree $p$. 
Since $\mathcal{S}_A \subset \mathcal{S}_B$, $P$ is easily seen to vanish on $\prod_{i=1}^d \Sigma_i$.
In addition, the monomials of degree $p$ in $P$ are the same as those of
\begin{align*}
 \left(\sum_{i=1}^d\lambda_i X_i\right)\left(\sum_{i=1}^dX_i\right)^{p-1}= & \left(\sum_{i=1}^d\lambda_i X_i\right)\left(\sum_{\substack{(t_1,\dots,t_d)\\\sum_{i=1}^d t_i=p-1}}\frac{(p-1)!}{\prod_{i=1}^d t_i!}\prod_{i=1}^d X_i^{t_i}\right)\\
 = & \sum_{i=1}^d \lambda_iX_i^p+\sum_{\substack{(t_1,\dots,t_d)\\\sum_{i=1}^d t_i=p \\ \max \{t_i\}<p}}\frac{(p-1)!}{\prod_{i=1}^d t_i!} \left(\sum_{i=1}^d\lambda_it_i\right)\prod_{i=1}^dX_i^{t_i}.
\end{align*}

Assume $\sum_{i=1}^d(|\Sigma_i|-1)>p$. 
Then there exists $(t_i)_{i \in [1,d]}$ such that $\sum^d_{i=1}t_i=p$, where $0 \le t_i \le |\Sigma_i|-1$ for all $i \in [1,d]$, and two distinct elements $i_0,j_0 \in [1,d]$ such that $t_{i_0} > 0 $ and $t_{j_0} < |\Sigma_{j_0}|-1$.
We define $(t'_i)_{i \in [1,d]}$ by $t'_{i_0}=t_{i_0}-1$, $t'_{j_0}=t_{j_0}+1$ and $t'_i=t_i$ otherwise.

Now, suppose that the coefficients of $\prod_{i=1}^dX_i^{t_i}$ and $\prod_{i=1}^dX_i^{t'_i}$ in $P$ both vanish. 
Then, both sums $\sum_{i=1}^d\lambda_it_i$ and $\sum_{i=1}^d\lambda_it'_i$ are zero, so that
$$\sum_{i=1}^d\lambda_it_i-\sum_{i=1}^d\lambda_it'_i=\lambda_{i_0}-\lambda_{j_0}=0,$$
which is a contradiction. 
Thus, one of these coefficients is nonzero, and by Theorem \ref{Combnull}, $P$ cannot vanish on $\prod_{i=1}^d \Sigma_i$, a contradiction too.
It follows that
$$\sum_{i=1}^d(|\Sigma_i|-1) \le p.$$
In addition, if one has $\sum_{i=1}^d(|\Sigma_i|-1)=p$, then since $P$ vanishes on $\prod_{i=1}^d \Sigma_i$, Theorem \ref{Combnull} implies that the coefficient of $\prod_{i=1}^dX_i^{|\Sigma_i|-1}$ in $P$ is zero, that is to say $\sum_{i=1}^d\lambda_i(|\Sigma_i|-1)=0 \text{ in } \mathbb{F}_p.$
\end{proof}

\subsection{The case of long sequences}
Thanks to Theorem \ref{prop1}, it is now easy to deduce our first theorem and its corollaries.

\begin{proof}[Proof of Theorem \ref{main theorem long >p}]
Let $A$ be a sequence of $\ell > p$ nonzero elements in $\mathbb{F}_p$.
For any given sequence $B$ of $\ell > p$ elements in $\mathbb{F}_p$ such that $\mathcal{S}_A\subset\mathcal{S}_B$,
let $\lambda_1,\dots,\lambda_d$ be the ratios associated with $(A,B)$ and assume $d \ge 2$.
On the one hand, the inequality (\ref{LB}) implies that $\sum_{i=1}^d(|\Sigma_i|-1)\ge \ell>p$. 
On the other hand, it follows from Theorem \ref{prop1} that $\sum_{i=1}^d(|\Sigma_i|-1)\le p$, which is a contradiction.
Therefore $d=1$, and the desired result is proved. 
\end{proof}

Corollary \ref{long >p characterisation} now is a straightforward consequence of Theorem \ref{main theorem long >p}.
In addition, one can easily notice that Corollary \ref{long >p minimales} follows directly from Theorem \ref{main theorem long >p} and the following general proposition. 

\begin{prop}
\label{mimimini} 
Let $p$  be a prime and let $A$ be a sequence of $\ell \ge 1$ nonzero elements in $\mathbb{F}_p$ such that $\dim(A)=\ell-1$.  
Then there exists a basis of $A^\perp$ consisting of minimal elements of $\mathcal{S}_A$.
\end{prop}

\begin{proof}
Let $(e_1,\dots,e_{\ell-1})$ be a basis of $A^\perp$ consisting of elements in $\mathcal{S}_A$. 
Let $F$ be the set of all minimal elements in $\mathcal{S}_A$ that are contained in at least one $e_i$.
Now, since each $e_i$ is a disjoint union of minimal elements in $\mathcal{S}_A$, we obtain $\langle \mathcal{S}_A \rangle \subset\langle F\rangle \subset A^\perp$. 
Since $\dim(A)=\ell-1$, it follows that $\langle F\rangle=A^\perp$ and $F$ contains a basis of $A^{\perp}$.
\end{proof}

\section{Sequences of length $p$}
\label{section : long p}

Let $p$ be a prime and let $\ell \le p$. 
In this section, we introduce a special class of sequences $A$ of $\ell$ nonzero elements in $\mathbb{F}_p$ for which $\dim(A) < \ell-1$ readily holds. 
In the case $\ell=p$, we determine all sequences in this class. 
We then refine the approach of Section \ref{section : long >p} to prove that $\dim(A)=p-1$ for all remaining sequences.

\subsection{Exceptional and regular sequences}
Let $A$ be a sequence of $\ell \le p$ nonzero elements in $\mathbb{F}_p$. 
Then, $A$ will be called \em exceptional \em whenever there exist two distinct elements $i,j \in [1,\ell]$ such that for all $x \in \mathcal{S}_A$, one has
$$|\{i,j\}\cap x|\in\{0,2\}.$$
Any sequence which is not exceptional will be called \em regular\em.
In others terms, $A$ is a regular sequence if and only if for any two distinct elements $i,j \in [1,\ell]$, there is $x \in \mathcal{S}_A$ such that $|\{i,j\}\cap x| = 1$.

\medskip
On the one hand, regular sequences have the following useful property.

\begin{lem}
\label{thecut} 
Let $p$ be a prime. 
Let $A$ be a regular sequence of $\ell$ elements in $\mathbb{F}^*_p$, and $B$ be a sequence of $\ell$ elements in $\mathbb{F}_p$ such that $\mathcal{S}_A\subset\mathcal{S}_B$. 
Then
$$\forall i,j \in [1,\ell], \, \, \, (a_i=a_j) \implies (b_i=b_j).$$
\end{lem}

\begin{proof}
Let $i,j$ be two distinct elements of $[1,\ell]$ such that $a_i=a_j$. 
Since $A$ is regular, there exists $x \in \mathcal{S}_A$ such that $|\{i,j\}\cap x| = 1$, say $i \in x$ and $j \notin x$.
Therefore, $x' = (x \smallsetminus \{i\}) \cup \{j\} \in \mathcal{S}_A$.
By assumption, it follows that $x,x' \in \mathcal{S}_B$ and $B \cdot x - B \cdot x'=b_i-b_j=0$, which is the desired result. 
\end{proof}

On the other hand, one can notice that $\dim(A) < \ell -1$ readily holds whenever $A$ is exceptional.
However, it turns out that for large values of $\ell$, exceptional sequences are highly structured and can easily be determined. 
For instance, the following lemma fully characterizes all exceptional sequences in the case $\ell=p$.

\begin{prop}
\label{degeneratecase l=p} 
Let $p$ be a prime and let $A=(a_1,\dots,a_p)$ be an exceptional sequence.
Then one of the following statements holds.
\begin{itemize}
\item[$(i)$] $\dim(A)=1$ and there exists $r \in \mathbb{F}^*_p$ such that
$$(a_1,\dots,a_p)=(r,\dots,r).$$
\item[$(ii)$] $\dim(A)=p-2$ and there exist $t \in [1,p-3]$, $\sigma\in\mathfrak{S}_p$, $r \in \mathbb{F}^*_p$ such that
$$(a_{\sigma(1)},\dots,a_{\sigma(p)})=(\underbrace{r,\dots,r}_{t},\underbrace{-r,\dots,-r}_{p-2-t},-(t+1)r,-(t+1)r).$$
\end{itemize}
\end{prop}

\begin{proof}
Since $A$ is exceptional, there exist two distinct elements $i,j \in [1,p]$ such that for all $x \in \mathcal{S}_A$, one has $|\{i,j\}\cap x|\in\{0,2\}$.
Now, let $A'$ be the sequence obtained from $A$ by deleting $a_i$ and $a_j$. 
Then, Cauchy-Davenport Theorem gives
$$\left|\Sigma(A')\right|\ge \min\{p,(p-2)+1\}=p-1.$$
By assumption, one has $-a_i \notin \Sigma(A')$ and $-a_j \notin \Sigma(A')$, so that $|\Sigma(A')|=p-1$ and $a_i=a_j$. 
Thus, by Lemma \ref{Vosper}, there exists $r\in \mathbb{F}^*_p$ such that $a_k \in \{\pm r\}$ for all $k \notin \{i,j\}$.
If $a_k=r$ for all $k \notin \{i,j\}$, then one has $-a_i=(p-1)r$, which gives $a_i=a_j=r$. 
Therefore, $A$ is of the form given in $(i)$, and it can easily be checked that $\dim(A)=p-1$.
Note that the same conclusion holds if $a_k=-r$ for all $k \notin \{i,j\}$.
Otherwise, there is $t \in [1,p-3]$ such that $A'$ consists of $t$ copies of $r$ and $p-2-t$ copies of $-r$.
In this case, we obtain $a_i=a_j=-(t+1)r$, so that, by relabelling if necessary, $A$ is of the form given in $(ii)$.
The following table then gives a basis of $\langle \mathcal{S}_A \rangle$.

\[\begin{array}{r|ccc|ccc|cc|}
\multicolumn{1}{c}{} & \multicolumn{3}{c}{\overbrace{\rule{1.5cm}{0pt}}^{t}} & \multicolumn{3}{c}{\overbrace{\rule{1.8cm}{0pt}}^{p-2-t}}\\
\cline{2-9}
& r & \dots & r & -r & \dots &-r & -(t+1)r & -(t+1)r\\
\cline{2-9}
\ldelim\{{3}{16mm}[$p-2-t$] & 1 &       &    & 1 &    &   (0)  &   & \\
& \vdots  & (0) &  &  &   \ddots  & &   &  \\
\ldelim\{{3}{4mm}[$t$] & 1 &   &  &  &  & 1 &(0) &(0) \\
 &  & \ddots &  &  & (0)  & \vdots &   &  \\
& (0) &      & 1  &  &  &  1  &  &  \\
\cline{2-9}
& 1 & \dots & 1 & 1 &  \dots & 1 & 1 & 1\\
\cline{2-9}
\end{array}\]
\end{proof}

\subsection{Another property of ratio sets}
Before proving Theorem \ref{main theorem long p}, we need the following general result on ratio sets, which will not only be useful in the case $\ell=p$, but in subsequent sections also.

\begin{prop}
\label{pushpull} 
Let $p$ be a prime. 
Let $A,B$ be two sequences of $\ell \ge 1$ elements in $\mathbb{F}_p$ such that the elements of $A$ are nonzero and $\mathcal{S}_A \subset \mathcal{S}_B$.
Let also $\lambda_1,\dots,\lambda_d$ be the ratios associated with $(A,B)$ and assume $d \ge 2$.
Given any two distinct elements $i_0,j_0 \in[1,d]$ and $r_{i_0}\in S_{i_0}$, define 
$$\Sigma'_{i_0}=\Sigma(S_{i_0}\smallsetminus(r_{i_0})),\ \Sigma'_{j_0}=\Sigma(S_{j_0}\cup(r_{i_0}))=\Sigma_{j_0}+\{0,r_{i_0}\}$$
and $\Sigma'_i=\Sigma_i$ otherwise.
Then the polynomial 
$$Q_{i_0,j_0}(X_1,\dots,X_d)=\left(\sum_{i=1}^d\lambda_i X_i\right)\left(\sum_{i=1}^d\lambda_i X_i-\chi\right)\left(\left(\sum_{i=1}^dX_i\right)^{p-1}-1\right),$$
where $\chi=(\lambda_{j_0}-\lambda_{i_0})r_{i_0}$, vanishes on $\prod_{i=1}^d\Sigma'_i$.

In addition, the coefficient of a monomial $\prod_{i=1}^dX_i^{t_i}$ in $Q_{i_0,j_0}$ with $\sum_{i=1}^dt_i=p+1$ and $\max\{t_i\}<p$ is
$$\frac{(p-1)!}{\prod_{i=1}^d t_i!}\left(\left(\sum_{i=1}^d\lambda_it_i\right)^2-\left(\sum_{i=1}^d\lambda^2_it_i\right)\right).$$
\end{prop}

\begin{proof}
Since $\mathcal{S}_A \subset \mathcal{S}_B$, the polynomial
$$P(X_1,\dots,X_d)=\left(\sum_{i=1}^d\lambda_i X_i\right)\left(\left(\sum_{i=1}^dX_i\right)^{p-1}-1\right)$$
vanishes on $\prod_{i=1}^d \Sigma_i$ and divides $Q_{i_0,j_0}$. 
Therefore, $Q_{i_0,j_0}$ vanishes on $\prod_{i=1}^d \Sigma_i$.

Now, for any $(x'_i)_{i \in [1,d]}\in\left(\prod_{i=1}^d \Sigma'_i\right)\smallsetminus\left(\prod_{i=1}^d \Sigma_i\right)$, 
let us consider $(x_i)_{i \in [1,d]}$ defined by $x_{i_0}=x'_{i_0}+r_{i_0}$, $x_{j_0}=x'_{j_0}-r_{i_0}$ and $x_i=x'_i$ otherwise, which is an element of $\prod_{i=1}^d \Sigma_i$.
Note that $\sum_{i=1}^dx'_i=\sum_{i=1}^dx_i$.
In addition, since $\mathcal{S}_A \subset \mathcal{S}_B$, the equality $\sum_{i=1}^dx'_i=0$ implies $\sum_{i=1}^d\lambda_ix_i=0$, so that
\begin{align*}
\sum_{i=1}^d\lambda_ix'_i = & \lambda_{i_0}(x_{i_0}-r_{i_0})+\lambda_{j_0}(x_{j_0}+r_{i_0})+\sum_{\substack{i=1\\i\neq i_0,\ i\neq j_0}}^d\lambda_ix_i\\
= & (\lambda_{j_0}-\lambda_{i_0})r_{i_0}+\sum_{i=1}^d\lambda_ix_i\\
= & (\lambda_{j_0}-\lambda_{i_0})r_{i_0},
\end{align*}
which proves that $Q_{i_0,j_0}$ vanishes on $\prod_{i=1}^d\Sigma'_i$.
Finally, $Q_{i_0,j_0}$ clearly has degree $p+1$ and the monomials of degree $p+1$ in $Q_{i_0,j_0}$ are the same as those of 

\begin{align*}
   & \left(\sum_{i=1}^d\lambda_i X_i\right)^2\left(\sum_{i=1}^dX_i\right)^{p-1}\\
 = & \left(\sum_{i=1}^d\lambda^2_i X^2_i+2\sum_{1\le i<j\le d}\lambda_i\lambda_j X_iX_j\right)\left(\sum_{\substack{(t_1,\dots,t_d)\\\sum_{i=1}^d t_i=p-1}}\frac{(p-1)!}{\prod_{i=1}^d t_i!}\prod_{i=1}^d X_i^{t_i}\right)\\
 = & \sum_{i=1}^d \lambda_i^2X_i^{p+1} + \sum_{i=1}^d \sum^d_{\substack{j=1 \\ j \neq i}}(-\lambda_i^2+2\lambda_i\lambda_j)X_i^pX_j\\
   & \hspace{0.5cm}+\sum_{\substack{(t_1,\dots,t_d)\\\sum_{i=1}^d t_i=p+1\\\max\{t_i\}<p}}\frac{(p-1)!}{\prod_{i=1}^d t_i!} \left(\sum_{i=1}^d\lambda_i^2t_i(t_i-1)+2\sum_{1\le i<j\le d}\lambda_i\lambda_j t_it_j\right)\prod_{i=1}^dX_i^{t_i}\\
 = & \sum_{i=1}^d \lambda_i^2X_i^{p+1}+ \sum_{i=1}^d \sum^d_{\substack{j=1 \\ j \neq i}}\lambda_i(2\lambda_j-\lambda_i)X_i^pX_j\\
   & \hspace{0.5cm}+\sum_{\substack{(t_1,\dots,t_d)\\\sum_{i=1}^d t_i=p+1\\\max\{t_i\}<p}}\frac{(p-1)!}{\prod_{i=1}^d t_i!} \left(\left(\sum_{i=1}^d\lambda_it_i\right)^2-\left(\sum_{i=1}^d\lambda^2_it_i\right)\right)\prod_{i=1}^dX_i^{t_i}.
\end{align*}
\end{proof}

\subsection{Proof of Theorem \ref{main theorem long p}}
We can now prove the main result of this section and its corollaries.

\begin{proof}[Proof of Theorem \ref{main theorem long p}]
Let $A$ be a sequence of $p$ nonzero elements in $\mathbb{F}_p$ not being of the forms given in $(i)$ or $(ii)$.
Thanks to Proposition \ref{degeneratecase l=p}, we can assume that $A$ is regular.
For any given sequence $B$ of $p$ elements in $\mathbb{F}_p$ such that $\mathcal{S}_A\subset\mathcal{S}_B$,
let $\lambda_1,\dots,\lambda_d$ be the ratios associated with $(A,B)$ and assume $d \ge 2$.
On the one hand, (\ref{LB}) implies that $\sum_{i=1}^d(|\Sigma_i|-1) \ge p$. 
On the other hand, it follows from Theorem \ref{prop1} that $\sum_{i=1}^d(|\Sigma_i|-1) \le p$,
which yields
$$\sum_{i=1}^d(|\Sigma_i|-1)= p \quad \text{and} \quad \sum_{i=1}^d\lambda_i(|\Sigma_i|-1)=0 \text{ in } \mathbb{F}_p.$$
Suppose that $d=2$. Since $|\Sigma_1|+|\Sigma_2|\ge p+2$, we obtain $|\Sigma_1 \cap (-\Sigma_2)| \ge 2$, 
which would contradict Lemma \ref{x-x}.
From now on, we thus assume $d\ge 3$. 
Since $A$ is regular, Lemmas \ref{x-x} and \ref{thecut} give $|\Sigma_i|<p-1$ for all $i \in [1,d]$. 
It follows that $|\Sigma_i|=|S_i|+1$ for all $i \in [1,d]$, and Lemma \ref{Vosper} implies that there exists $r_i\in \mathbb{F}^*_p$ such that all elements of $S_i$ are equal to $r_i$ or $-r_i$.

For any two distinct elements $i_0,j_0 \in [1,d]$ and any $r_{i_0} \in S_{i_0}$, let us consider the sets $\Sigma'_i$ defined by
$$\Sigma'_{i_0}=\Sigma(S_{i_0}\smallsetminus(r_{i_0})),\ \Sigma'_{j_0}=\Sigma(S_{j_0}\cup(r_{i_0}))=\Sigma_{j_0}+\{0,r_{i_0}\}$$
and $\Sigma'_i=\Sigma_i$ otherwise. Let also $Q_{i_0,j_0}$ be as in Proposition \ref{pushpull}.

Since $\Sigma_{i_0}$ is an arithmetic progression with difference $r_{i_0}$, one has $|\Sigma'_{i_0}|=|\Sigma_{i_0}|-1$.  
In addition, since $A$ is regular, Lemmas \ref{x-x} and \ref{thecut} imply that $r_{i_0} \neq \pm r_{j_0}$. It follows from Lemma \ref{Vosper} that
$|\Sigma'_{j_0}|\ge|\Sigma_{j_0}|+2$.

We now define $(t_i)_{i \in [1,d]}$ by  $t_{i_0}=|\Sigma_{i_0}|-2$, $t_{j_0}=|\Sigma_{j_0}|+1$ and $t_i=|\Sigma_i|-1$ otherwise.
In particular, $t_i \le |\Sigma'_i|-1 < p$ for all $i \in [1,d]$ and $\sum^d_{i=1}t_i=p+1$.
Since $\sum_{i=1}^d\lambda_i(|\Sigma_i|-1)=0$ in $\mathbb{F}_p$, Proposition \ref{pushpull} implies that, up to a nonzero multiplicative constant, the coefficient of $\prod_{i=1}^d X_i^{t_i}$ in $Q_{i_0,j_0}$ is

\begin{align*}
\left(\sum_{i=1}^d\lambda_it_i\right)^2-\left(\sum_{i=1}^d\lambda^2_it_i\right) = & (2\lambda_{j_0}-\lambda_{i_0})^2-\left(2\lambda_{j_0}^2-\lambda_{i_0}^2+\sum_{i=1}^d\lambda_i^2(|\Sigma_i|-1)\right)\\
= & (2\lambda_{j_0}^2-4\lambda_{i_0}\lambda_{j_0}+2\lambda_{i_0}^2)-\sum_{i=1}^d\lambda_i^2(|\Sigma_i|-1)\\
= & 2(\lambda_{j_0}-\lambda_{i_0})^2-\sum_{i=1}^d\lambda_i^2(|\Sigma_i|-1).
\end{align*}
Note that $\sum_{i=1}^d\lambda_i^2(|\Sigma_i|-1)$ is independent of $i_0$ and $j_0$. 
In addition, the coefficient of $\prod_{i=1}^d X_i^{t_i}$ in $Q_{i_0,j_0}$ is zero if and only if 
$$(\lambda_{j_0}-\lambda_{i_0})^2=\frac{1}{2}\sum_{i=1}^d\lambda_i^2(|\Sigma_i|-1).$$
Since $d\ge 3$, there is at least one pair $(i_0,j_0)$ such that this equality does not hold. 
It follows from Theorem \ref{Combnull} that for this actual pair, $Q_{i_0,j_0}$ cannot vanish on $\prod_{i=1}^d \Sigma'_i$, which contradicts Proposition \ref{pushpull}.
\end{proof}

\begin{proof}[Proof of Corollary \ref{reconstruction l=p}]
Let $A,B$ be two sequences of $p$ nonzero elements in $\mathbb{F}_p$.
One direction of the implication is trivial. 
Thus, we shall prove that $A$ and $B$ are collinear whenever $\mathcal{S}_A=\mathcal{S}_B$.
By Theorem \ref{main theorem long p}, there are three cases to consider.
If $\dim(A)=p-1$, then $\mathcal{S}_A \subset \mathcal{S}_B$ already implies the required result.
If $A$ is constant, it is also easily seen that $\mathcal{S}_A = \mathcal{S}_B$ holds if and only if $B$ is constant.
Otherwise, $A$ and $B$ both are of the form given in Theorem \ref{main theorem long p} $(ii)$. 
In particular, there exist $t \in [1,p-3]$ and $r \in \mathbb{F}^*_p$ such that, by relabelling if necessary, one has
$$A=(\underbrace{r,\dots,r}_{t},\underbrace{-r,\dots,-r}_{p-2-t},-(t+1)r,-(t+1)r).$$
Since $\mathcal{S}_A \subset \mathcal{S}_B$, we deduce that
$$B=(\underbrace{\lambda r,\dots,\lambda r}_{t},\underbrace{-\lambda r,\dots,- \lambda r}_{p-2-t},-(t+1)\lambda r + d,-(t+1)\lambda r -d),$$
for some $\lambda \in \mathbb{F}^*_p$ and $d \in \mathbb{F}_p$. 
Thus, $B$ is of the form given in Theorem \ref{main theorem long p} $(ii)$ if and only if $d=0$, which completes the proof.
\end{proof}

\begin{proof}[Proof of Corollary \ref{long p minimales}]
Let $A$ be a nonconstant sequence of $p$ nonzero elements in $\mathbb{F}_p$. 
By Theorem \ref{main theorem long p}, there are two cases to consider. 
If $\dim(A)=p-1$, the required result directly follows from Proposition \ref{mimimini}. 
Otherwise, there exist $t \in [1,p-3]$ and $r \in \mathbb{F}^*_p$ such that $A$ consists of $t$ copies of $r$, $p-2-t$ copies of $-r$ and two copies of $-(t+1)r$.
Denoting the minimum and maximum of $t$ and $p-2-t$ by $m$ and $M$ respectively, it is easily seen that the number of minimal elements in $\mathcal{S_A}$ is
$$mM+{M \choose m}.$$
Now, if $t \in [2,p-4]$, then $p \ge 7$ and $\mathcal{S}_A$ contains at least $2(p-4) \ge p-1$ minimal elements.
If $t \in \{1,p-3\}$, then $p \ge 5$ and $\mathcal{S}_A$ contains exactly $2(p-3) \ge p-1$ minimal elements, which completes the proof.
\end{proof}

\section{Sequences of length $p-1$}
\label{section : long p-1}

Let $p$ be prime and let $A$ be a sequence of $p-1$ nonzero elements in $\mathbb{F}_p$.
In this section, we start by showing that the dimension and structure of $A$ can easily be deduced from Theorem \ref{main theorem long p} whenever $\sigma(A) \neq 0$.
Then, we concentrate on the case where $\sigma(A)=0$ and prove Theorem \ref{main theorem long p-1} as follows. 
On the one hand, we determine all exceptional zero-sum sequences. 
On the other hand, we further refine our general approach to handle the case of regular zero-sum sequences.

\subsection{The case of nonzero-sum sequences}
Given a sequence $A=(a_1,\dots,a_\ell)$ of $\ell \ge 1$ nonzero elements in $\mathbb{F}_p$ such that $\sigma(A) \neq 0$, we consider the sequence $A'=(a_1,\dots,a_\ell,-\sigma(A))$.
The aim of the following lemma is to show that the value of $\dim(A)$ can easily be derived from $\dim(A')$.

\begin{lem} 
\label{dim of non zero-sum sequences}
Let $p$ be a prime and let $A$ be a sequence of $\ell \ge 1$ nonzero elements in $\mathbb{F}_p$ such that $\sigma(A) \neq 0$. 
Then $A'$ is a zero-sum sequence and
$$\dim(A')=\dim(A)+1.$$
\end{lem}

\begin{proof}
Let us set $X=\{(x_1,\dots,x_{\ell+1}) \in \mathcal{S}_{A'} : x_{\ell+1} = 0 \}$. 
We clearly have $\sigma(A')=0$, which implies that $\mathcal{S}_{A'}$ is closed under complement. 
Therefore,
$$\langle \mathcal{S}_{A'} \rangle = \langle X \cup \{(1,\dots,1)\} \rangle.$$
In addition, it follows from the very definition of $X$ that $(1,\dots,1) \in \mathcal{S}_{A'} \setminus \langle X \rangle$ and $\dim(X)=\dim(A)$. Thus,
$$\dim(A')=\dim(\langle \mathcal{S}_{A'} \rangle)=\dim(X) +1 = \dim(A)+1,$$
which completes the proof.
\end{proof}

In particular, Lemma \ref{dim of non zero-sum sequences} implies that for every sequence $A$ of $p-1$ nonzero elements in $\mathbb{F}_p$ such that $\sigma(A) \neq 0$, either $\dim(A)=p-2$, or $\dim(A)=0$ and $A$ is constant, or $\dim(A)=p-3$ and $A$ can be obtained by deleting any element in a sequence of type $(ii)$ in Theorem \ref{main theorem long p}.

\smallskip
The reconstruction problems on $A$ and $A'$ are also closely related to each other.

\begin{lem} 
\label{reconstruction of non zero-sum sequences}
Let $p$ be a prime and let $A,B$ be two sequences of $\ell \ge 1$ nonzero elements in $\mathbb{F}_p$ such that $\sigma(A) \neq 0$.
Then $\mathcal{S}_{A'}=\mathcal{S}_{B'}$ if and only if $\mathcal{S}_{A}=\mathcal{S}_{B}$. 
\end{lem}

\begin{proof} 
Since $\sigma(A')=0$, the desired result is a straightforward consequence of the fact that $\mathcal{S}_{A'}$ is closed under complement.
\end{proof}

For instance, let $A,B$ be two sequences of $p-1$ nonzero elements in $\mathbb{F}_p$ such that $\sigma(A) \neq 0$ and $\mathcal{S}_A=\mathcal{S}_B$.
Specifying $\ell=p-1$ in Lemma \ref{reconstruction of non zero-sum sequences}, we obtain that $\mathcal{S}_{A'}=\mathcal{S}_{B'}$. 
Then, it easily follows from Corollary \ref{reconstruction l=p} that $A'$ and $B'$ are collinear, and so are $A$ and $B$.
From now on, we thus consider zero-sum sequences only.

\subsection{The case of exceptional zero-sum sequences}
We now state and prove the following lemma, which fully characterizes all exceptional zero-sum sequences in the case $\ell=p-1$.

\begin{prop}\label{degeneratecase l=p-1}
Let $p$ be a prime and let $A=(a_1,\dots,a_{p-1})$ be an exceptional zero-sum sequence. 
Then one of the following statements holds.
\begin{itemize}
\item[$(i)$] $\dim(A)=1$ and there exist $\sigma\in\mathfrak{S}_{p-1}$, $r \in \mathbb{F}^*_p$ such that
$$(a_{\sigma(1)},\dots,a_{\sigma(p-1)})=(r,\dots,r,2r).$$
\item[$(ii)$] $\dim(A)=p-4$ and there exist $\sigma\in\mathfrak{S}_{p-1}$, $r \in \mathbb{F}^*_p$ such that
$$(a_{\sigma(1)},\dots,a_{\sigma(p-1)})=(\underbrace{r,\dots,r}_{p-5},-r,2r,2r,2r).$$
\item[$(iii)$] $\dim(A)=p-3$ and there exist $t\in[0,p-6]$, $\sigma\in\mathfrak{S}_{p-1}$, $r \in \mathbb{F}^*_p$ such that
$$(a_{\sigma(1)},\dots,a_{\sigma(p-1)})=(\underbrace{r,\dots,r}_{t},\underbrace{-r,\dots,-r}_{p-4-t},2r,-(t+3)r,-(t+3)r).$$
\item[$(iv)$] $\dim(A)=p-3$ and there exist $t\in[1,p-4]$, $\sigma\in\mathfrak{S}_{p-1}$, $r \in \mathbb{F}^*_p$ such that
$$(a_{\sigma(1)},\dots,a_{\sigma(p-1)})=(\underbrace{r,\dots,r}_{t},\underbrace{-r,\dots,-r}_{p-3-t},-(t+1)r,-(t+2)r).$$
\end{itemize}
\end{prop}

\begin{proof}
Since $A$ is exceptional, there exist two distinct elements $i,j \in [1,p-1]$ such that for all $x \in \mathcal{S}_A$, one has $|\{i,j\}\cap x|\in\{0,2\}$. 
Now, let $A'$ be the sequence obtained from $A$ by deleting $a_i$ and $a_j$.
Then, Cauchy-Davenport Theorem gives
$$\left|\Sigma(A')\right| \ge \min\{p,(p-3)+1\} = p-2.$$
By assumption, one has $-a_i \notin \Sigma(A')$ and $-a_j \notin \Sigma(A')$.
We now consider two cases.

\smallskip
$\bullet$ 
$\left|\Sigma(A')\right|=p-2$.
Then, by Lemma \ref{Vosper}, there exist $t \in [0,p-3]$ and $r\in \mathbb{F}^*_p$ such that $A'$ consists of $t$ copies of $r$ and $p-3-t$ copies of $-r$.
Therefore, we obtain that $\{a_i,a_j\} \subset \{-(t+1)r,-(t+2)r\}$. 
If $a_i = a_j$, then $a_i+a_j \in \{-2(t+1)r,-2(t+2)r\}$. 
Since $\sigma(A')=(2t+3)r$, we would have $\sigma(A) \neq 0$, which is a contradiction.
Thus, one has $a_i \neq a_j$ and, by relabelling if necessary, it follows that either $t \in \{0,p-3\}$ and $A$ is of the form given in $(i)$, 
or $t \in [1,p-4]$ and $A$ is of the form given in $(iv)$. 

\smallskip
$\bullet$ 
$\left|\Sigma(A')\right|=p-1$.
Then, $a_i=a_j$ and by Lemma \ref{theinv}, there exist $t \in [0,p-4]$ and $r \in \mathbb{F}^*_p$ such that $A'$ consists of $t$ copies of $r$, $p-4-t$ copies of $-r$ and one copy of $2r$.
If $t=p-4$, then $\Sigma(A')= \mathbb{F}_p \smallsetminus\{-r\}$ and $a_i=a_j=r$, so that $A$ is of the form given in $(i)$. 
If $t=p-5$, then $\Sigma(A')= \mathbb{F}_p \smallsetminus\{-2r\}$, $a_i=a_j=2r$ and $A$ is of the form given in $(ii)$.
Otherwise, one has $t \in [0,p-6]$, $\Sigma(A')= \mathbb{F}_p \smallsetminus\{(t+3)r\}$, $a_i=a_j=-(t+3)r$, and $A$ is of the form given in $(iii)$. 

\medskip
The following tables give a basis of $\langle \mathcal{S}_A \rangle$ for cases $(ii)$ to $(iv)$.
\begin{itemize}
\item[$(ii)$] $\dim(A)=p-4$.
\[\begin{array}{r|ccc|c|ccc|}
\multicolumn{1}{c}{} & \multicolumn{3}{c}{\overbrace{\rule{1.7cm}{0pt}}^{p-5}}\\
\cline{2-8}
&r & \dots & r & -r & 2r & 2r & 2r\\
\cline{2-8}
\ldelim\{{3}{10mm}[$p-5$] &1 &       & (0)   & 1 &    &   &   \\
&  &\ddots &   & \vdots  & & (0)  &  \\
& (0)  &   &  1 & 1  &  &  &  \\
\cline{2-8}
&1 & \dots & 1 & 1 & 1 & 1 & 1\\
\cline{2-8}
\end{array}\]

\item[$(iii)$] $\dim(A)=p-3$ and $t\in[1,p-6]$.
\[\begin{array}{r|ccc|ccc|c|cc|}
\multicolumn{1}{c}{} & \multicolumn{3}{c}{\overbrace{\rule{1.5cm}{0pt}}^{t}} & \multicolumn{3}{c}{\overbrace{\rule{1.8cm}{0pt}}^{p-4-t}}\\
\cline{2-10}
&r & \dots & r & -r & \dots &-r & 2r & -(t+3)r & -(t+3)r\\
\cline{2-10}
\ldelim\{{3}{16mm}[$p-4-t$] & 1 &       &    & 1 &    & (0)  &  &  & \\
 & \vdots  & (0) &   &  &   \ddots &  & &   &  \\
\ldelim\{{3}{4mm}[$t$] &  1 &   &   &   &  & 1 &(0) &(0)  & (0) \\
 & & \ddots &  &  &(0)  & \vdots &  & &  \\
&(0) &      & 1  &  &  &  1  &  & & \\
\cline{2-10}
&  &  (0)    &   & 1 &(0)  & 1  & 1 & 0 & 0 \\
\cline{2-10}
& 1 & \dots & 1 & 1 &  \dots & 1 & 1 & 1 & 1\\
\cline{2-10}
\end{array}\]
And whenever $t=0$,
\[\begin{array}{r|cccc|c|cc|}
\multicolumn{1}{c}{} & \multicolumn{4}{c}{\overbrace{\rule{2.8cm}{0pt}}^{p-4}}\\
  \cline{2-8}
& -r & -r & \dots &-r & 2r & -3r & -3r\\
\cline{2-8}
\ldelim\{{3}{10mm}[$p-5$] &  1 &  1 &   & (0)  & 1 &  & \\
& \vdots &  &   \ddots &  & \vdots & (0)  & (0) \\
& 1  & (0) &  & 1 & 1 &  &  \\
 \cline{2-8}
& 0   & 1 &(0) &1  & 1 & 0 & 0\\
 \cline{2-8}
&  1 & 1 & \dots & 1 & 1 & 1 & 1\\
 \cline{2-8}
\end{array}\]

\item[$(iv)$] $\dim(A)=p-3$ and $t\in[1,p-4]$.
\[\begin{array}{r|ccc|ccc|cc|}
\multicolumn{1}{c}{} & \multicolumn{3}{c}{\overbrace{\rule{1.5cm}{0pt}}^{t}} & \multicolumn{3}{c}{\overbrace{\rule{1.8cm}{0pt}}^{p-3-t}}\\
\cline{2-9}
& r & \dots & r & -r & \dots &-r & -(t+1)r & -(t+2)r\\
\cline{2-9}
\ldelim\{{3}{16mm}[$p-3-t$] & 1 &       &   & 1 &    &  (0)  &  & \\
& \vdots  &(0) &   &  &   \ddots  & &   &  \\
\ldelim\{{3}{4mm}[$t$] & 1 &   &   &   &  & 1 &(0)  & (0) \\
 & & \ddots &  &  &(0)  & \vdots &   &  \\
&(0) &      & 1  &  &  &  1  &  &  \\
\cline{2-9}
& 1 & \dots & 1 & 1 &  \dots & 1 & 1 & 1\\
\cline{2-9}
\end{array}\]
\end{itemize}
\end{proof}

\subsection{Preliminary results}

In the proof of Theorem \ref{main theorem long p-1}, we will need the following two lemmas to deal with the cases where the number of distinct ratios is small.
The first one gives some nonvanishing properties of a particular quadratic polynomial.

\begin{lem}
\label{conic} 
Let $p \ge 5$ be a prime. Given $u,v \in \mathbb{F}_p$, let $f$ be the quadratic polynomial
$$f(x,y)=2x^2-6xy+6y^2+2u(3y-x)+v.$$
Then, for any three pairwise distinct elements $\alpha_1, \alpha_2,\alpha_3 \in \mathbb{F}_p$, 
there exists $\sigma \in \mathfrak{S}_3$ such that one of the following holds.
\begin{itemize}
\item[$(i)$] $f(\alpha_{\sigma(1)},\alpha_{\sigma(2)}) \, f(\alpha_{\sigma(1)},\alpha_{\sigma(3)})\neq 0$,
\item[$(ii)$] $f(\alpha_{\sigma(2)},\alpha_{\sigma(1)}) \, f(\alpha_{\sigma(3)},\alpha_{\sigma(1)})\neq 0$,
\item[$(iii)$] $f(\alpha_{\sigma(1)},\alpha_{\sigma(2)}) \, f(\alpha_{\sigma(2)},\alpha_{\sigma(1)})\neq 0$,
\item[$(iv)$] $f(\alpha_{\sigma(1)},\alpha_{\sigma(2)}) \, f(\alpha_{\sigma(2)},\alpha_{\sigma(3)}) \, f(\alpha_{\sigma(3)},\alpha_{\sigma(1)})\neq 0$.
\end{itemize}
\end{lem}

\begin{proof}
It is an easy exercise to show that, up to a nonzero multiplicative constant, the only quadratic polynomial vanishing on the six points 
$(\alpha_1,\alpha_2)$, $(\alpha_2,\alpha_1)$, $(\alpha_1,\alpha_3)$, $(\alpha_3,\alpha_1)$, $(\alpha_2,\alpha_3)$ and $(\alpha_3,\alpha_2)$ is
$$x^2+y^2+xy-(x+y)(\alpha_1+\alpha_2+\alpha_3)+(\alpha_1\alpha_2+\alpha_1\alpha_3+\alpha_2\alpha_3).$$

Thus, there exist two distinct elements $i,j \in \{1,2,3\}$ such that $f(\alpha_i,\alpha_j) \neq 0$. 
Now, let $k$ be the remaining element of $\{1,2,3\}$ and assume that for all $\sigma \in \mathfrak{S}_3$, none of the properties $(i)$ to $(iv)$ holds. 
Then one has $f(\alpha_j,\alpha_i)=0$, $f(\alpha_i,\alpha_k)=0$ and $f(\alpha_k,\alpha_j)=0$. 
In addition, since $f(\alpha_i,\alpha_j) \, f(\alpha_j,\alpha_k) \, f(\alpha_k,\alpha_i)=0$, we have that either $f(\alpha_j,\alpha_k)=0$ or $f(\alpha_k,\alpha_i)=0$.
Therefore, there exists $\sigma \in \mathfrak{S}_3$ such that
$$f\left(\alpha_{\sigma(1)},\alpha_{\sigma(2)}\right)=f(\alpha_{\sigma(2)},\alpha_{\sigma(1)})=f(\alpha_{\sigma(1)},\alpha_{\sigma(3)})=f(\alpha_{\sigma(3)},\alpha_{\sigma(2)})=0.$$
Without loss of generality, we may assume that $\sigma$ is the identity, so that
\begin{align*}
f(\alpha_1,\alpha_2)-f(\alpha_2,\alpha_1)=0 & \iff -4(\alpha_1-\alpha_2)(\alpha_1+\alpha_2+2u)=0\\
f(\alpha_1,\alpha_2)-f(\alpha_1,\alpha_3)=0 & \iff 6(\alpha_2-\alpha_3)(-\alpha_1+\alpha_2+\alpha_3+u)=0\\
f(\alpha_1,\alpha_2)-f(\alpha_3,\alpha_2)=0 & \iff 2(\alpha_1-\alpha_3)(\alpha_1-3\alpha_2+\alpha_3-u)=0.
\end{align*}
Since $\alpha_1,\alpha_2,\alpha_3$ are pairwise distinct, they are solution to the linear system
$$\left\{\begin{array}{rcl}
\alpha_1+\alpha_2\phantom{+5\alpha_3} & = & -2u\\
-\alpha_1+\alpha_2+\alpha_3 & = & -u\\
\alpha_1-3\alpha_2+\alpha_3 & = & u.
\end{array}\right.$$
Since $p \ge 5$, the determinant of this system is $6 \neq 0$ in $\mathbb{F}_p$. 
Therefore, this system has a unique solution, which is easily seen to be $\alpha_1=\alpha_2=\alpha_3=-u$.
This contradicts the fact that $\alpha_1,\alpha_2,\alpha_3$ are pairwise distinct, and the proof is complete.
\end{proof}

Our second lemma gives the structure of a pair of arithmetic progressions covering almost all $\mathbb{F}_p$, but intersecting in only one element.

\begin{lem}
\label{P1P2} 
Let $p$ be a prime. 
Let also $P_1$ and $P_2$ be two arithmetic progressions in $\mathbb{F}_p$ with distinct differences, such that $\min\{|P_1|,|P_2|\}\ge 3$ and $P_1\cap P_2=\{0\}$.

\smallskip
$\bullet$ If $|P_1|+|P_2|=p$, then there exists $r \in \mathbb{F}^*_p$ such that $P_1$ and $P_2$ have one of the following forms.
\begin{enumerate}
\item $\{0,2,4\}.r$ and $\{5,6,\dots,p-1,0,1\}.r$,
\item $\{-\frac{p-1}{2},0,\frac{p-1}{2}\}.r$ and $\{-\frac{p-5}{2},-\frac{p-7}{2},\dots,\frac{p-5}{2},\frac{p-3}{2}\}.r$.
\end{enumerate}

\smallskip
$\bullet$ If $|P_1|+|P_2|=p-1$, then there exists $r \in \mathbb{F}^*_p$ such that $P_1$ and $P_2$ have one of the following forms.
\begin{enumerate}
\setcounter{enumi}{2}
\item $\{0,2,4\}.r$ and $\{5,6,\dots,p-1,0\}.r$,
\item $\{0,2,4\}.r$ and $\{6,7,\dots,p-1,0,1\}.r$,
\item $\{0,2,4,6\}.r$ and $\{7,8,\dots,p-1,0,1\}.r$,
\item $\{0,3,6\}.r$ and $\{7,8,\dots,0,1,2\}.r$,
\item $\{-\frac{p-1}{2},0,\frac{p-1}{2}\}.r$ and $\{-\frac{p-5}{2},-\frac{p-7}{2},\dots,\frac{p-7}{2},\frac{p-5}{2}\}.r$,
\item $\{-\frac{p-1}{2},0,\frac{p-1}{2}\}.r$ and $\{-\frac{p-7}{2},-\frac{p-9}{2},\dots,\frac{p-5}{2},\frac{p-3}{2}\}.r$,
\item $\{-\frac{p-3}{2},0,\frac{p-3}{2}\}.r$ and $\{-\frac{p-5}{2},-\frac{p-7}{2},\dots,\frac{p-7}{2},\frac{p-5}{2}\}.r$,
\item $p=11$ and
\begin{enumerate}
\item $\{-6,-3,0,3,6\}.r$ and $\{-2,-1,0,1,2\}.r$,
\item $\{-4,0,4,8\}.r$ and $\{-2,-1,0,1,2,3\}.r$,
\item $\{-8,-4,0,4,8\}.r$ and $\{-2,-1,0,1,2\}.r$,
\end{enumerate}
\item $p=13$ and $\{-8,-4,0,4,8\}.r$ and $\{-3,-2,-1,0,1,2,3\}.r$.
\end{enumerate}
\end{lem}

\begin{proof}
Without loss of generality, we can suppose that $|P_1| \ge |P_2|$. Since $|P_1|+|P_2|\ge p-1$, one has $|P_1| \ge (p-1)/2$.
Up to a nonzero multiplicative constant, we can also assume that $P_1$ has difference $1$.
Now, we set 
$$P_2=\{-\beta d,\dots,-d,0,d,\dots,\alpha d\},$$
where $\alpha$ and $\beta$ are two nonnegative integers.
Note that replacing $(d,\alpha,\beta)$ by $(-d,\beta,\alpha)$ yields the same arithmetic progression, so we may assume that $d \in [2,(p-1)/2]$. 
In addition, $P_1$ and $P_2$ satisfy the hypothesis of the lemma if and only if $-P_1$ and $-P_2$ do so, which allows us to suppose that $d \in P_2$, that is $\alpha \ge 1$.
Since $P_1\cap P_2=\{0\}$, we obtain that  
$$P_1 = \{p-(|P_1|-(t+1)),\dots,p-1,0,1,\dots,t\},$$
where $0 \le t < d$. 

We consider first the case where $\beta = 0$.
Then, $P_2=\{0,d,\dots,\alpha d\}$ and $\alpha \ge 2$.
Also, there exists a unique pair of integers $(q,u)$ such that
$$p-(|P_1|-(t+1))=qd+u \quad \text{ and } \quad 1 \le u \le d.$$
Since $P_1\cap P_2=\{0\}$, we have $qd+u \ge \alpha d+1$. 
Now, counting separately the missing elements of $P_1 \cup P_2$ in the three intervals $[0,d]$, $[d+1,\alpha d-1]$ and $[\alpha d,p-1]$, we obtain that
$$|\mathbb{F}_p \setminus (P_1\cup P_2)|=((d-1)-t)+(\alpha-1)(d-1)+((qd+u)-(\alpha d +1)).$$
Since $\alpha,d \ge 2$, one has $(\alpha-1)(d-1) \ge 1$. 
We now distinguish two cases. 
\begin{itemize}
\item In the case $|P_1|+|P_2|=p$, one has $|P_1\cup P_2|=p-1$, so that $\alpha=d=q=2$ and $u=t=1$. 
This is the structure of case $(1)$ in the statement of the lemma.
\item In the case $|P_1|+|P_2|=p-1$, one has $|P_1\cup P_2|=p-2$, so that $(qd+u)-(\alpha d +1) \le 1$, which implies $\alpha=q$.
\begin{itemize}
\item If $q=2$ and $d=2$, one has either $u=1$ and $t=0$, this is case $(3)$, or $u=2$ and $t=1$, this is case $(4)$.
\item If $q=3$ and $d=2$, one has $u=1$ and $t=1$, this is case $(5)$.
\item If $q=2$ and $d=3$, one has $u=1$ and $t=2$, this is case $(6)$.
\end{itemize}
\end{itemize}

We now turn to the case where $\beta \ge 1$. 
Since $P_1 \subset [-d+1,d-1]$, one has 
$$\frac{p-1}{2} \le |P_1| \le 2d-1.$$
\begin{itemize}
\item If $|P_2|=3$, then either $|P_1|=p-3$ and $d=(p-1)/2$, which is case (2), or $|P_1|=p-4$ and $d \in \{(p-3)/2,(p-1)/2\}$, which yields cases $(7),(8),(9)$.
\item If $|P_2| \ge 4$, then $\alpha \ge 2$ or $\beta \ge 2$.
Therefore, one of the two intervals $[d+1,2d-1]$ and $[-2d+1,-d-1]$ has to be disjoint from $P_1$.
Such an interval can contain at most one element of $P_2$, and thus contains at least $d-2$ elements being neither in $P_1$ nor in $P_2$.
Since $|P_1|+|P_2|\ge p-1$ and $|P_1\cap P_2|=1$, one has $|P_1\cup P_2|\ge p-2$ and then $d-2\le 2$.
\begin{itemize}
\item If $d=2$, then $4 \le |P_2| \le |P_1| \le 3$, a contradiction.
\item If $d=3$, then $p \le 11$ and $4 \le |P_2| \le |P_1| \le 5$.
Since $|P_1|+|P_2| \ge 8$, it follows that $p=11$ and $|P_1|=|P_2|=5$, which gives (10.a). 
\item If $d=4$, then $p \le 15$ and $4 \le |P_2| \le |P_1| \le 7$.
Since $|P_1|+|P_2| \ge 8$, it follows that either $p=11$ and (10.b) or (10.c) holds, or $p=13$ and (11) holds.
\end{itemize}
\end{itemize}
\end{proof}

\subsection{Proof of Theorem \ref{main theorem long p-1}}
We can now prove the last of our main theorems.

\begin{proof}[Proof of Theorem \ref{main theorem long p-1}]
Let $A=(a_1,\dots,a_{p-1})$ be a zero-sum sequence of $p-1$ nonzero elements in $\mathbb{F}_p$ not being of the forms given in $(i)$ to $(iv)$.
Thanks to Proposition \ref{degeneratecase l=p-1}, we can assume that $A$ is regular.
For any given sequence $B$ of $p-1$ elements in $\mathbb{F}_p$ such that $\mathcal{S}_A\subset\mathcal{S}_B$,
let $\lambda_1,\dots,\lambda_d$ be the ratios associated with $(A,B)$ and assume $d \ge 2$.
On the one hand, the inequality (\ref{LB}) implies that $\sum_{i=1}^d(|\Sigma_i|-1) \ge p-1$. 
On the other hand, it follows from Theorem \ref{prop1} that $\sum_{i=1}^d(|\Sigma_i|-1) \le p$, which yields
$$\sum_{i=1}^d(|\Sigma_i|-1)\in\{ p-1,p\}.$$

\smallskip
\textbf{Case 1: $d=2$.}

\smallskip
By Lemma \ref{x-x}, one has $0\notin(\Sigma_1\smallsetminus\{0\})+\Sigma_2$ and Cauchy-Davenport Theorem implies that
$$|(\Sigma_1\smallsetminus\{0\})+\Sigma_2| \ge (|\Sigma_1|-1)+(|\Sigma_2|-1) \ge |S_1|+|S_2|=p-1.$$
Therefore, one has $|\Sigma_i|=|S_i|+1$ for each $i \in \{1,2\}$. 
Since $|\Sigma_i| < p$ by Lemma \ref{x-x}, it follows from Lemma \ref{Vosper} that there exists $r_i \in \mathbb{F}^*_p$ such that all elements of $S_i$ are equal to $r_i$ or $-r_i$.
In addition, since $A$ is a zero-sum sequence, Lemma \ref{x-x} yields $\sigma(S_1)=\sigma(S_2)=0$, so that $r_i$ occurs as many times as $-r_i$ in $S_i$.

Now, for each $i \in \{1,2\}$, consider the subsequence $S'_i$ of $S_i$ which contains only the occurrences of $r_i$, and set $\Sigma'_i=\Sigma(S'_i)$. 
We clearly have that
$$|S'_1|+|S'_2|=\frac{p-1}{2} \quad \text{ and } \quad |\Sigma'_1|+|\Sigma'_2|=\frac{p+3}{2}.$$
Suppose that there exist $(x_1,y_1)\in (\Sigma'_1)^2$ and $(x_2,y_2)\in (\Sigma'_2)^2$ such that $x_1+x_2=y_1+y_2$. 
Then, one has $x_1-y_1=y_2-x_2$, where $x_1-y_1\in \Sigma_1$ and $x_2-y_2\in \Sigma_2$. 
Thus, Lemma \ref{x-x} implies that $x_1=y_1$ and $x_2=y_2$. 
This proves that all elements of $\Sigma'_1+\Sigma'_2$ have a unique representation 
as a sum $x_1+x_2$, where $(x_1,x_2) \in  \Sigma'_1 \times \Sigma'_2$, so that
$$|\Sigma'_1+\Sigma'_2|=|\Sigma'_1||\Sigma'_2|=|\Sigma'_1|\left(\frac{p+3}{2}-|\Sigma'_1|\right).$$

If $2 < |\Sigma'_1| < \frac{p-1}{2}$, one has $|\Sigma'_1+\Sigma'_2|>2\left(\frac{p-1}{2}\right)=p-1$. 
In addition, since $p$ is prime, one cannot have $|\Sigma'_1+\Sigma'_2|=|\Sigma'_1||\Sigma'_2|=p$ either.
Therefore, one of the two sets, say $\Sigma'_1$, has cardinality $2$.
Then $|S'_1|=1$ and $S_1=(r_1,-r_1)$.
Now, since $A$ is regular, there is a subsequence $T_2$ of $S_2$ such that  $\sigma(S'_1)+\sigma(T_2)=0$.
It follows from Lemma \ref{x-x} that $\sigma(S'_1)=r_1=0$, which is a contradiction.

\smallskip
\textbf{Case 2:} $d \ge 3$ and $\sum_{i=1}^d(|\Sigma_i|-1)=p$.

\smallskip
On the one hand, since $d \ge 3$ and $A$ is regular, Lemmas \ref{x-x} and \ref{thecut} imply that $|\Sigma_i| < p-1$ for all $i \in [1,d]$. 
On the other hand, there is $k \in [1,d]$ such that $|\Sigma_i|=|S_i|+1$ for all $i \neq k$. 
Using Lemma \ref{Vosper}, we obtain that for every $i \neq k$, there exists $r_i\in \mathbb{F}^*_p$ such that all elements of $S_i$ are equal to $r_i$ or $-r_i$.

Now, for any two distinct elements $i_0, j_0 \in [1,d]$ with $i_0 \neq k$, and any $r_{i_0} \in S_{i_0}$,
let us consider the sets $\Sigma'_i$ defined by
$$\Sigma'_{i_0}=\Sigma(S_{i_0}\smallsetminus(r_{i_0})),\ \Sigma'_{j_0}=\Sigma(S_{j_0}\cup(r_{i_0}))=\Sigma_{j_0}+\{0,r_{i_0}\}$$
and $\Sigma'_i=\Sigma_i$ otherwise. Let also $Q_{i_0,j_0}$ be as in Proposition \ref{pushpull}.

Then, one has $|\Sigma'_{i_0}|=|\Sigma_{i_0}|-1$.  
In addition, since $A$ is regular, Lemmas \ref{x-x} and \ref{thecut} imply that $\pm r_{i_0} \notin S_{j_0}$. 
It follows from Theorem \ref{ThB} that $|\Sigma'_{j_0}|\ge|\Sigma_{j_0}|+2$.

We now define $(t_i)_{i \in [1,d]}$ by  $t_{i_0}=|\Sigma_{i_0}|-2$, $t_{j_0}=|\Sigma_{j_0}|+1$ and $t_i=|\Sigma_i|-1$ otherwise.
In particular, $t_i \le |\Sigma'_i|-1 < p$ for all $i \in [1,d]$ and $\sum^d_{i=1}t_i=p+1$.
By Theorem \ref{prop1}, we have that $\sum_{i=1}^d\lambda_i(|\Sigma_i|-1)=0$ in $\mathbb{F}_p$. 
Thus, Proposition \ref{pushpull} implies that, up to a nonzero multiplicative constant, the coefficient of $\prod_{i=1}^d X_i^{t_i}$ in $Q_{i_0,j_0}$ is

\begin{align*}
\left(\sum_{i=1}^d\lambda_it_i\right)^2-\left(\sum_{i=1}^d\lambda^2_it_i\right) = & (2\lambda_{j_0}-\lambda_{i_0})^2-\left(2\lambda_{j_0}^2-\lambda_{i_0}^2+\sum_{i=1}^d\lambda_i^2(|\Sigma_i|-1)\right)\\
= & (2\lambda_{j_0}^2-4\lambda_{i_0}\lambda_{j_0}+2\lambda_{i_0}^2)-\sum_{i=1}^d\lambda_i^2(|\Sigma_i|-1)\\
= & 2(\lambda_{j_0}-\lambda_{i_0})^2-\sum_{i=1}^d\lambda_i^2(|\Sigma_i|-1).
\end{align*}
Note that $\sum_{i=1}^d\lambda_i^2(|\Sigma_i|-1)$ is independent of $i_0$ and $j_0$. 
In addition, the coefficient of $\prod_{i=1}^d X_i^{t_i}$ in $Q_{i_0,j_0}$ is zero if and only if 
$$(\lambda_{j_0}-\lambda_{i_0})^2=\frac{1}{2}\sum_{i=1}^d\lambda_i^2(|\Sigma_i|-1).$$
Since $d\ge 3$, there is at least one pair $(i_0,j_0)$ such that this equality does not hold. 
Since $i_0$ and $j_0$ play symmetric roles, $i_0$ can be chosen such that $i_0 \neq k$ indeed.
It follows from Theorem \ref{Combnull} that for this actual pair, $Q_{i_0,j_0}$ cannot vanish on $\prod_{i=1}^d \Sigma'_i$, which contradicts Proposition \ref{pushpull}.

\smallskip
\textbf{Case 3:} $d \ge 3$ and $\sum_{i=1}^d(|\Sigma_i|-1)=p-1$.

\smallskip
Since $A$ is regular, Lemmas \ref{x-x} and \ref{thecut} give $|\Sigma_i|<p-1$ for all $i \in [1,d]$. 
It follows that $|\Sigma_i|=|S_i|+1$ for all $i \in [1,d]$, and Lemma \ref{Vosper} implies that there exists $r_i\in \mathbb{F}^*_p$ such that all elements of $S_i$ are equal to $r_i$ or $-r_i$.
Note, in particular, that one has $\pm r_i \neq r_j$ for any two distinct elements $i,j \in [1,d]$.

In addition, if one of the ratio sets, say $\Sigma_1$, has cardinality $p-2$, then $d=3$ and $|S_2|=|S_3|=1$.
Since $A$ is regular, there would be a subsequence $T_1$ of $S_1$ such that  $\sigma(T_1)+\sigma(S_2)=0$. 
By Lemma \ref{x-x}, we would have $\sigma(S_2)=0$, a contradiction. 
This proves that $|\Sigma_i|<p-2$ for all $i \in [1,d]$.

Now, for any two distinct elements $i_0, j_0 \in [1,d]$ and any $r_{i_0} \in S_{i_0}$ such that $|S_{j_0}| > 1$ and $\pm \frac{1}{2}r_{i_0} \notin S_{j_0}$,
let us consider the sets $\Sigma'_i$ defined by
$$\Sigma'_{i_0}=\Sigma(S_{i_0}\smallsetminus(r_{i_0})),\ \Sigma'_{j_0}=\Sigma(S_{j_0}\cup(r_{i_0}))=\Sigma_{j_0}+\{0,r_{i_0}\}$$
and $\Sigma'_i=\Sigma_i$ otherwise. 
Let also $Q_{i_0,j_0}$ be as in Proposition \ref{pushpull}.

Then, one has $|\Sigma'_{i_0}|=|\Sigma_{i_0}|-1$.  
In addition, since $|S_{j_0}| > 1$ and neither $\pm r_{i_0}$ nor $\pm \frac{1}{2}r_{i_0}$ are elements of $S_{j_0}$, it follows from Lemma \ref{theinv} that $|\Sigma'_{j_0}|\ge|\Sigma_{j_0}|+3$.

We now define $(t_i)_{i \in [1,d]}$ by  $t_{i_0}=|\Sigma_{i_0}|-2$, $t_{j_0}=|\Sigma_{j_0}|+2$ and $t_i=|\Sigma_i|-1$ otherwise.
In particular, $t_i \le |\Sigma'_i|-1 < p$ for all $i \in [1,d]$ and $\sum^d_{i=1}t_i=p+1$.
Proposition \ref{pushpull} implies that, up to a nonzero multiplicative constant, the coefficient of $\prod_{i=1}^d X_i^{t_i}$ in $Q_{i_0,j_0}$ is

\begin{align*}
 & \left(\sum_{i=1}^d\lambda_it_i\right)^2-\left(\sum_{i=1}^d\lambda^2_it_i\right)\\
 = &\left (\sum_{i=1}^d\lambda_i(|\Sigma_i|-1)+3\lambda_{j_0}-\lambda_{i_0}\right)^2-\left(\sum_{i=1}^d\lambda_i^2(|\Sigma_i|-1)+3\lambda_{j_0}^2-\lambda_{i_0}^2\right)\\
= & (2\lambda_{i_0}^2-6\lambda_{i_0}\lambda_{j_0}+6\lambda_{i_0}^2)+2\left(\sum_{i=1}^d\lambda_i(|\Sigma_i|-1)\right)(3\lambda_{j_0}-\lambda_{i_0})\\
 & \hspace{4cm}+\left(\sum_{i=1}^d\lambda_i(|\Sigma_i|-1)\right)^2-\left(\sum_{i=1}^d\lambda_i^2(|\Sigma_i|-1)\right).
\end{align*}
Since $\sum_{i=1}^d\lambda_i(|\Sigma_i|-1)$ and $\sum_{i=1}^d\lambda_i^2(|\Sigma_i|-1)$ are independent of $i_0$ and $j_0$, this last expression is a quadratic polynomial $f$ in $\lambda_{i_0}$ and $\lambda_{j_0}$. 
In addition, the coefficient of $\prod_{i=1}^d X_i^{t_i}$ in $Q_{i_0,j_0}$ is zero if and only if $f(\lambda_{i_0},\lambda_{j_0})=0$.

\smallskip
\textbf{Case 3.1:} If $d\ge 5$, then choose any $j_0 \in [1,d]$ such that $|S_{j_0}|>1$. 
On the one hand, there are at most two ratios $\lambda_i$ such that $f(\lambda_i,\lambda_{j_0})=0$. 
On the other hand, there is at most one set $S_i$ such that $\pm \frac{1}{2}r_{i} \in S_{j_0}$.
Since $d\ge 5$, the pigeonhole principle implies that there is at least one pair $(i_0,j_0)$ such that $f(\lambda_{i_0},\lambda_{j_0}) \neq 0$ and $\pm \frac{1}{2}r_{i_0} \notin S_{j_0}$.
It follows from Theorem \ref{Combnull} that for this actual pair, $Q_{i_0,j_0}$ cannot vanish on $\prod_{i=1}^d \Sigma'_i$, which contradicts Proposition \ref{pushpull}.

\smallskip
\textbf{Case 3.2:} If $d\in\{3,4\}$ and there are at least three sets $S_i$, such that $|S_i|>1$,
then without loss of generality, we may assume that $|S_i| > 1$ for each $i \in \{1,2,3\}$. 
In particular, $p \ge 7$. 
Now, by Lemma \ref{conic}, one of the following holds.

\smallskip
\begin{itemize}
\item 
If $f(\lambda_1,\lambda_2) \, f(\lambda_1,\lambda_3)\neq 0$, then there is $j_0 \in \{2,3\}$ such that $\pm\frac{1}{2}r_1 \notin S_{j_0}$. 
Therefore, setting $i_0=1$, we have $f(\lambda_{i_0},\lambda_{j_0}) \neq 0$ and $\pm \frac{1}{2}r_{i_0} \notin S_{j_0}$ indeed. 

\smallskip
\item 
If $f(\lambda_2,\lambda_1) \, f(\lambda_3,\lambda_1)\neq 0$, then there is $i_0 \in \{2,3\}$ such that $\pm\frac{1}{2}r_{i_0} \notin S_1$.
Therefore, setting $j_0=1$, we have $f(\lambda_{i_0},\lambda_{j_0}) \neq 0$ and $\pm \frac{1}{2}r_{i_0} \notin S_{j_0}$ indeed.

\smallskip
\item  
If $f(\lambda_1,\lambda_2) \, f(\lambda_2,\lambda_1)\neq 0$, then $p \ge 7$ implies that one cannot have both $r_1=\pm 2r_2$ and $r_2=\pm 2r_1$. 
Thus, there are two distinct elements $i_0,j_0 \in \{1,2\}$ such that $f(\lambda_{i_0},\lambda_{j_0}) \neq 0$ and $\pm \frac{1}{2}r_{i_0} \notin S_{j_0}$ indeed.

\smallskip
\item 
If $f(\lambda_1,\lambda_2) \, f(\lambda_2,\lambda_3) \, f(\lambda_3,\lambda_1)\neq 0$, then $p \ge 7$ implies that the three equalities $r_1=\pm 2r_2$ and $r_2=\pm 2r_3$ and $r_3=\pm 2r_1$ can occur only when $r_i=\pm 8r_i$, so that $p=7$. 
Otherwise, there are two distinct elements $i_0,j_0 \in \{1,2,3\}$ such that $f(\lambda_{i_0},\lambda_{j_0}) \neq 0$ and $\pm \frac{1}{2}r_{i_0} \notin S_{j_0}$ indeed.
\end{itemize}

It follows that whenever $p \neq 7$, there is at least one pair $(i_0,j_0)$ such that $f(\lambda_{i_0},\lambda_{j_0}) \neq 0$ and $\pm \frac{1}{2}r_{i_0} \notin S_{j_0}$.
Theorem \ref{Combnull} then implies that for this actual pair, $Q_{i_0,j_0}$ cannot vanish on $\prod_{i=1}^d \Sigma'_i$, which contradicts Proposition \ref{pushpull}.

\smallskip
The only remaining case to consider is $p=7$, $d=3$ and $|S_i|=2$ for each $i \in \{1,2,3\}$. 
Therefore, by relabelling if necessary, there are only two possible cases. 
\begin{itemize}
\item $\Sigma_1=\{0,1,2\}$, $\Sigma_2=\{0,2,4\}$, $\Sigma_3=\{0,4,1\}$, which gives $S_1=(1,1)$, $S_2=(2,2)$ and $S_3=(-3,-3)$. 
It can easily be checked that such a sequence $A$ satisfies $\dim(A)=p-2$, so that $d=1$, a contradiction. 
\item $\Sigma_1=\{-1,0,1\}$, $\Sigma_2=\{-2,0,2\}$, $\Sigma_3=\{-3,0,3\}$, which gives $S_1=(-1,1)$, $S_2=(-2,2)$ and $S_3=(-3,3)$. 
Therefore, by relabelling if necessary, one has $A=(-1,1,-2,2,-3,3 )$. 
The following table then gives a basis of $\langle \mathcal{S}_A \rangle$.

$$\begin{array}{|cc|cc|cc|}
\hline
 -1 & 1 & -2 & 2 & -3 & 3 \\
\hline
1 &  1   &  0  & 0 &  0  & 0  \\
 0 & 0 & 1  & 1  & 0 & 0 \\
 0  & 0 & 0 & 0  & 1 & 1 \\
 1 & 0 & 1 & 0 & 0 & 1 \\
\hline
\end{array}$$
\end{itemize}

\smallskip
\textbf{Case 3.3: } If $d\in\{3,4\}$ and there are exactly two sets $S_i$ such that $|S_i|>1$, 
then without loss of generality, we can suppose that these two sets are $S_1$ and $S_2$, where $|S_1| \le |S_2|$. 
Therefore $\Sigma_1$ and $\Sigma_2$ are two arithmetic progressions such that $|\Sigma_1|+|\Sigma_2|=p+3-d$. 
In addition, it follows from Lemma \ref{x-x} that $\Sigma_1$ and $-\Sigma_2$ satisfy the hypothesis of Lemma \ref{P1P2}.
We thus consider the following cases. 

\smallskip
If $d=3$, then $p \ge 7$ and since $\sigma(S_1)+\sigma(S_2)+\sigma(S_3)=0$, Lemma \ref{x-x} implies that $\sigma(S_1)$ and $\sigma(S_2)$ are nonzero. 
Therefore, Lemma \ref{P1P2} implies there exists $r \in \mathbb{F}^*_p$ such that $\Sigma_1 = \{0,2,4\}.r$ and $\Sigma_2=\{-1,0,1,\dots,p-5\}.r$. 
On the one hand, this gives $\sigma(S_1)+\sigma(S_2)=-2r$. 
On the other hand, the single element of $S_3$ has to be in $\{3,4\}.r$, so that $\sigma(A) \neq 0$, a contradiction.

\smallskip
If $d=4$, then $p \ge 11$ and since $\sigma(S_1)+\sigma(S_2)+\sigma(S_3)+\sigma(S_4)=0$, Lemma \ref{x-x} implies that either $\sigma(S_1)$ or $\sigma(S_2)$ is nonzero. 
Therefore, Lemma \ref{P1P2} implies there exists $r \in \mathbb{F}^*_p$ such that one of the following cases holds. 

\begin{itemize}
\item $\Sigma_1=\{0,2,4\}.r$ and $\Sigma_2=\{0,1,\dots,p-5\}.r$. 
Then, $\sigma(S_1)+\sigma(S_2)=-r$ and the elements of $S_3$ and $S_4$ have to be in $\{3,4\}.r$. 
Therefore, one has $\sigma(A) \neq 0$, which is a contradiction.
\item $\Sigma_1=\{0,2,4\}.r$ and $\Sigma_2=\{-1,0,1,\dots,p-6\}.r$. 
Then, $\sigma(S_1)+\sigma(S_2)=-3r$ and the elements of $S_3$ and $S_4$ have to be in $\{3,4,5\}.r$.
Therefore, one has $\sigma(A) \neq 0$, which is a contradiction.
\item $\Sigma_1=\{0,2,4,6\}.r$ and $\Sigma_2=\{-1,0,1,\dots,p-7\}.r$. 
Then, $\sigma(S_1)+\sigma(S_2)=-2r$ and the elements of $S_3$ and $S_4$ have to be in $\{3,4,6\}.r$ when $p=11$, and in $\{3,4,5,6\}.r$ whenever $p\ge 13$.
In all cases, one has $\sigma(A) \neq 0$, which is a contradiction.
\item $\Sigma_1=\{0,3,6\}.r$ and $\Sigma_2=\{-2,-1,0,\dots,p-7\}.r$. 
Then, $\sigma(S_1)+\sigma(S_2)=-3r$ and the elements of $S_3$ and $S_4$ have to be in $\{4,6\}.r$ when $p=11$, and in $\{4,5,6\}.r$ whenever $p\ge 13$.
In all cases, one has $\sigma(A) \neq 0$, which is a contradiction.
\item 
$\Sigma_1=\{-\frac{p-1}{2},0,\frac{p-1}{2}\}.r$ and $\Sigma_2=\{-\frac{p-3}{2},\dots,-1,0,1,\dots,\frac{p-7}{2}\}.r$
Then, $\sigma(S_1)+\sigma(S_2)=-2r$ 
and the elements of $S_3$ and $S_4$ have to be in $\{\frac{p+3}{2},\frac{p+5}{2}\}.r$.
Therefore, one has $\sigma(A) \neq 0$, which is a contradiction.
\item $p=11$,
$\Sigma_1=\{-4,0,4,8\}.r$ and $\Sigma_2=\{-2,-1,0,1,2,3\}.r$. 
Then, the elements of $S_3$ and $S_4$ have to be in $\{-5,5\}.r$, which would contradict Lemma \ref{x-x}.
\end{itemize}

\smallskip
\textbf{Case 3.4: } If $d\in\{3,4\}$ and there is only one set $S_i$ such that $|S_i|>1$, then without loss of generality, we can suppose that this set is $S_1$. 
If $d=3$, and since $A$ is regular, there is a subsequence $T_1$ of $S_1$ such that  $\sigma(T_1)+\sigma(S_2)=0$. 
It follows from Lemma \ref{x-x} that $\sigma(S_2)=0$, which is a contradiction.
Thus, $d=4$ and Lemma \ref{x-x} implies that, by relabelling if necessary, there exist $t \in [0,p-4]$ and $r \in \mathbb{F}^*_p$ such that
\[\Sigma_1=\{-(p-4-t),\dots,-1,0,1,\dots,t\}.r,\]
\[\Sigma_2=\{0,-(t+1)\}.r,\ \Sigma_3=\{0,-(t+2)\}.r,\ \Sigma_4=\{0,-(t+3)\}.r.\]
Since $t \in [0,p-4]$, one has $\sigma(A)=-(t+2)r \neq 0$, which contradicts the fact that $A$ is a zero-sum sequence. 
\end{proof}

\begin{proof}[Proof of Corollary \ref{reconstruction long p-1}]
Let $A,B$ be two sequences of $p-1$ nonzero elements in $\mathbb{F}_p$.
It is easily checked that $\mathcal{S}_A=\mathcal{S}_B$ whenever $A$ and $B$ are collinear, or have one of the forms given in $(i)$, $(ii)$ and $(iii)$. 
Conversely, suppose that $A,B$ are such that $\mathcal{S}_A=\mathcal{S}_B$. 
If $\sigma(A)\neq 0$ then Lemma \ref{reconstruction of non zero-sum sequences} and Corollary \ref{reconstruction l=p} imply that $A$ and $B$ are collinear.
If $\sigma(A)=0$ then, using Theorem  \ref{main theorem long p-1}, there are six cases to consider. 

\smallskip
$\bullet$ If $\dim(A)=p-2$ then $\mathcal{S}_A \subset \mathcal{S}_B$ already implies that $A$ and $B$ are collinear.

\smallskip
$\bullet$ If $\dim(A)=1$, then $\dim(B)=1$ and it is easily seen that there exist $\lambda,r \in \mathbb{F}^*_p$ such that $A$ (resp. $B$) consists of $p-2$ copies of $r$ (resp. $\lambda r$) and one copy of $2r$ (resp. $2\lambda r$). Therefore, $A$ and $B$ either are collinear or have the form given in $(i)$. 

\smallskip
$\bullet$ If $p=7$ and, by relabelling if necessary, one has 
$$A=(-1,1,-2,2,-3,3),$$ 
then it can be checked by hand that $\mathcal{S}_A = \mathcal{S}_B$ if and only there is $\lambda \in \mathbb{F}^*_p$ such that 
$$B=(-\lambda,\lambda,-2\lambda,2\lambda,-3\lambda,3\lambda) \quad \text{ or } \quad B=(-\lambda,\lambda,3\lambda,-3\lambda,2\lambda,-2\lambda).$$
Therefore, $A$ and $B$ either are collinear or have the form given in $(iii)$.

\smallskip
$\bullet$ If $\dim(A)=p-3$ and there exist $t\in[1,p-4],r \in \mathbb{F}^*_p$ such that, by relabelling if necessary, one has
$$A=(\underbrace{r,\dots,r}_{t},\underbrace{-r,\dots,-r}_{p-3-t},-(t+1)r,-(t+2)r),$$
then $\mathcal{S}_A \subset \mathcal{S}_B$ implies that
$$B=(\underbrace{\lambda r,\dots,\lambda r}_{t},\underbrace{-\lambda r,\dots,-\lambda r}_{p-3-t},-(t+1)\lambda r+d,-(t+2)\lambda r-d)$$
for some $\lambda \in \mathbb{F}^*_p$ and $d \in \mathbb{F}_p$. 
By Theorem \ref{main theorem long p-1}, the equality $\dim(B)=p-3$ holds only when $d=0$, that is $A$ and $B$ are collinear, or $d=-\lambda r$, which gives case $(ii)$. 

\smallskip
$\bullet$ If $\dim(A)=p-3$ and there exist $t\in[0,p-6],r \in \mathbb{F}^*_p$ such that, by relabelling if necessary, one has
$$A=(\underbrace{r,\dots,r}_{t},\underbrace{-r,\dots,-r}_{p-4-t},2r,-(t+3)r,-(t+3)r),$$
then $\mathcal{S}_A \subset \mathcal{S}_B$ implies that
$$B=(\underbrace{\lambda r,\dots,\lambda r}_{t},\underbrace{-\lambda r,\dots,-\lambda r}_{p-4-t},2\lambda r,-(t+3)\lambda r+d,-(t+3)\lambda r-d)$$
for some $\lambda \in \mathbb{F}^*_p$ and $d \in \mathbb{F}_p$. 
By Theorem \ref{main theorem long p-1}, the equality $\dim(B)=p-3$ holds only when $d=0$, that is $A$ and $B$ are collinear. 

\smallskip
$\bullet$ If $\dim(A)=p-4$, then there exist $r \in \mathbb{F}^*_p$ such that, by relabelling if necessary, one has
$$A=(\underbrace{r,\dots,r}_{p-5},-r,2r,2r,2r).$$
Now, either one has $p=5$, which brings us back to the case $\dim(A)=1$, or $p \ge 7$ and $\mathcal{S}_A \subset \mathcal{S}_B$ implies that
$$B=(\underbrace{\lambda r,\dots,\lambda r}_{p-5},-\lambda r,2\lambda r+a,2\lambda r+b,2\lambda r-(a+b))$$
for some $\lambda \in \mathbb{F}^*_p$ and $a,b \in \mathbb{F}_p$. 
Then, by Theorem \ref{main theorem long p-1}, the equality $\dim(B) = p-4$ holds only when $a=b=0$, so that $A$ and $B$ are collinear. 
\end{proof}

\begin{proof}[Proof of Corollary \ref{long p-1 minimales}]
Let $A$ be a sequence of $p-1$ nonzero elements in $\mathbb{F}_p$ not being of the forms given in $(i)$ and $(ii)$.
If $\sigma(A) \neq 0$, then Lemma \ref{dim of non zero-sum sequences} and Theorem \ref{main theorem long p} readily imply that $\dim(A) \ge p-3$.
Then, the desired result follows directly from Proposition \ref{mimimini}.
If $\sigma(A)= 0$ then, using Theorem \ref{main theorem long p-1}, there are only two cases to consider. 
If $\dim(A) \ge p-3$, then Proposition \ref{mimimini} implies that the number of minimal elements in $\mathcal{S}_A$ is at least $p-3$.
If $\dim(A)=p-4$, then there exist $r \in \mathbb{F}^*_p$ such that $A$ consists of $p-5$ copies of $r$, one copy of $-r$ and three copies of $2r$.
Now, either one has $p=5$ and $A$ has the form given in $(ii)$, or $p \ge 7$ and the number of minimal elements in $\mathcal{S}_A$ is $2(p-5) \ge p-3$. 
\end{proof}

\section{A concluding remark}
\label{section : conclusion}

Let $p$ be a prime and let $A=(a_1,\dots,a_\ell)$ be a sequence of $\ell \ge 1$ nonzero elements in $\mathbb{F}_p$.
In this paper, we proved that for every $\ell \ge p-1$, the equality $\dim(A)=\ell-1$ holds except for a very limited number of exceptional sequences which can be fully determined.
However, our results can easily be extended to the affine setting.

\smallskip
Given any element $\alpha \in \mathbb{F}_p$, let $\mathcal{S}^\alpha_A$ be the set of all $0$-$1$ solutions to the equation
$$a_1x_1 + \dots + a_\ell x_\ell = \alpha,$$
and let 
$$\dim(A,\alpha)=\dim(\mathrm{aff}(\mathcal{S}^\alpha_A))$$
be the dimension of the affine hull of $\mathcal{S}_A^\alpha$.

\smallskip
For any $I \subset [1,\ell]$, we consider the sequence $A_I=(a'_1,\dots,a'_\ell)$ defined by $a'_i=-a_i$ if $i\in I$ and $a'_i=a_i$ otherwise.
Whenever $\ell\ge p-1$, the following lemma shows that there is an affine transformation mapping $\mathcal{S}^\alpha_A$ onto $\mathcal{S}_{A_I}$, for a particular $I \subset [1,\ell]$.
Therefore, our results provide a full description of the dimension and structure of the sets $\mathcal{S}^\alpha_A$, for all $\alpha \in \mathbb{F}_p$. 

\begin{lem}
Let $p$ be a prime and let $A=(a_1,\dots,a_\ell)$ be a sequence of $\ell \ge p-1$ nonzero elements in $\mathbb{F}_p$. 
Then, for every $\alpha \in \mathbb{F}_p$, there exists $I\subset [1,\ell]$ such that
$$\dim(A,\alpha)=\dim(A_I).$$
\end{lem}

\begin{proof} 
By Cauchy-Davenport Theorem, one has $|\Sigma(A)|\ge\min\{p,\ell+1\}=p$. 
Thus, there exists $I\subset[1,\ell]$ such that $\sum_{i\in I}a_i=\alpha$.
Now, let $A_I=(a'_1,\dots,a'_\ell)$.
Then, for every $J\subset[1,\ell]$, one has  
$$\sum_{i\in J}a_i =\alpha$$
if and only if 
$$\sum_{i\in J}a_i-\sum_{i\in I}a_i=0$$
which is equivalent to   
$$\sum_{i\in J\smallsetminus I}a_i-\sum_{i\in I\smallsetminus J}a_i=0$$
that is to say  
$$\sum_{i\in I \Delta J}a'_i=0.$$
Therefore, $\mathcal{S}_{A_I}$ is the image of $\mathcal{S}^\alpha_A$ by the affine transformation $(x_1,\dots,x_\ell)\mapsto(y_1,\dots y_\ell)$, where $y_i=1-x_i$ if $i\in I$ and $y_i=x_i$ otherwise.
\end{proof}

\end{document}